\documentclass{amsart}
\usepackage[top=50pt,bottom=57pt,left=76pt,right=76pt]{geometry}

\usepackage{amsfonts, mathrsfs, enumerate, amsmath, amsthm, amssymb,
thm-restate, pifont, mathrsfs}
\usepackage{textcomp}

\usepackage{tikzit,xcolor,color}
\usetikzlibrary{shapes.multipart}
\tikzstyle{new edge style 2}=[-, thick]

\usepackage[colorlinks]{hyperref}
\hypersetup{linkcolor=blue, urlcolor=blue, citecolor=red}

\usepackage{theoremref}

\newcommand{\ljy}[1]{\textcolor{green}{#1}}
\newcommand{\ignore}[1]{}

\newtheorem{theorem}{Theorem}[section]

\newtheorem{cor}[theorem]{Corollary}
\newtheorem{prop}[theorem]{Proposition}
\newtheorem{lem}[theorem]{Lemma}
\newtheorem{definition}[theorem]{Definition}
\newtheorem{claim}[theorem]{Claim}
\newtheorem{question}[theorem]{Question}
\newtheorem{example}[theorem]{Example}
\newenvironment{ex}{\begin{example} \rm}{\end{example}}
\newenvironment{de}{\begin{definition} \rm}{\end{definition}}

\newcommand{\N}{\mathbb{N}}
\newcommand{\R}{\mathbb{R}}
\newcommand{\T}{\mathcal{T}}

\newcommand{\set}[2]{\{ #1\colon#2 \} }  
  
\newcommand{\Inj}{\operatorname{Inj}}  
  
\newcommand{\Sym}{\operatorname{Sym}}  
\newcommand{\im}{\operatorname{im}}
\newcommand{\id}{\operatorname{id}}

\newcommand{\Q}{\mathbb{Q}}

\renewcommand{\to}{\longrightarrow}
\newcommand{\dom}{\operatorname{dom}}
\newcommand{\na}{\mathbin{\blacklozenge}}

\newcommand{\Aut}{\operatorname{Aut}}
\newcommand{\End}{\operatorname{End}}
\newcommand{\Emb}{\operatorname{Emb}}
\newcommand{\Poly}{\Pol}
\newcommand{\Pol}{\operatorname{Pol}}

\newcommand{\EndQ}{\End(\Q,\leq)}

\newcommand{\Age}{\operatorname{Age}}
\newcommand{\x}{\mathbf{x}}

\newcommand{\reduct}[2]{\mathbb{#1}(\mathbb{#2})}

\title[Polish semigroup topologies on endomorphism monoids]{ Polish topologies on endomorphism monoids of relational structures}
\author{L. Elliott, J. Jonu\v{s}as, J. D. Mitchell, Y. P\'eresse, and M. Pinsker}
\thanks{L. Elliott would like to acknowledge the support of Mathematics and
Statistics at the University of St Andrews for supporting their Ph.D. studies.
J. Jonu\v{s}as received funding from the Austrian Science Fund (FWF) through
Lise Meitner grant No M 2555. M. Pinsker has received
funding from the Austrian Science Fund (FWF) through project No P32337
and from the Czech Science Foundation (grant No 18-20123S)}
\begin{document}
\maketitle

\begin{abstract}
In this paper we present general techniques for characterising minimal and maximal semigroup topologies on the endomorphism monoid $\End(\mathbb{A})$ of a  countable relational structure $\mathbb{A}$. As applications, we show that the endomorphism monoids of several well-known relational structures,  including the random graph, the random directed graph, and the random partial order, possess a unique Polish semigroup topology. In every case this unique topology is the subspace topology induced by the usual topology on the Baire space $\N ^ \N$. 
 We also show that many of these structures  have the property that every homomorphism from their endomorphism monoid to a second countable topological semigroup is continuous; referred to as \textit{automatic continuity}. 
Many of the results about endomorphism monoids are extended to clones of polymorphisms on the same structures.
\end{abstract}

\section{Introduction}\label{section-intro}

The Baire space $\N ^ \N$ is naturally endowed with the product topology arising from the discrete topology on every copy of $\N$. This topology will be referred to as the \textit{pointwise topology}. A subbasis for the pointwise topology on $\N ^ \N$ consists of the sets $U_{x, y} = \set{f\in \N^ \N}{(x)f = y}$ for all $x,y \in \N$.
The topological space $\N ^ \N$ also forms a monoid with operation the composition of functions $\circ$, called the \textit{full transformation monoid on $\N$}. The space $\N ^ \N$ is central in the context of Polish spaces (completely metrizable and separable topological spaces), and in the theory of semigroups. For example, every non-empty Polish space is a continuous image of the Baire space, and every countable semigroup embeds into the full transformation monoid. 
The pointwise topology happens to be compatible with the algebraic structure of $\N ^ \N$ in the sense that the function
$\circ : \N ^ \N \times \N ^ \N\to \N ^ \N$ is continuous; such a topology is referred to as a \textit{semigroup topology} on $\N ^ \N$.  In fact, 
the algebraic structure of $\N ^ \N$ and the pointwise topology are deeply intertwined: the pointwise topology is the unique Polish semigroup topology on $\N ^ \N$~\cite[Theorem 5.4]{Elliott2019aa}. 

The monoid $\N ^ \N$ is not alone in having a unique Polish topology compatible with its algebraic structure. 
As a $G_{\delta}$ subset of $\N ^ \N$, the symmetric group $\Sym(\N)$ is also a Polish space with the subspace topology induced by the pointwise topology on $\N ^ \N$. Abusing notation slightly, we will  refer to the subspace topology on any subset of $\N^\N$, induced by the pointwise topology on $\N ^ \N$, as the pointwise topology. The pointwise topology on $\Sym(\N)$ is likewise compatible with the group structure of $\Sym(\N)$, in that multiplication and inversion are continuous. In~\cite{Gaughan1967aa} it was shown that every Hausdorff group topology on $\Sym(\N)$ contains the pointwise topology. It is folklore in the theory of Polish groups that if $\T_1$ and $\T_2$ are Polish group topologies on a group $G$ and $\T_1\subseteq \T_2$, then $\T_1 = \T_2$; see \cite[Theorem 9.10 and Proposition 11.5]{Kechris1995aa}. It follows that the pointwise topology is the unique Polish group topology on $\Sym(\N)$. Uniqueness of compatible topologies has been studied for many further groups and semigroups, and more general objects such as clones; see for example~\cite{Behrisch2017aa, Bodirsky2018aa, Bodirsky2017aa, Chang2017aa, Cohen2016aa, Elliott2019aa, Gartside2008ab,
Hodges1993ab, Hrushovski1992aa, Kallman1976aa, Kallman1979aa, Kallman1984aa, Kallman1984ab, Kallman1986aa, Kallman2010aa, Kechris2007aa, Pech2016aa, Pech2017aa, Pech:2018aa, Perez2020aa, Rosendal2007ac, Sabok2019aa, Shelah1984aa, Solovay1970aa}.

Every closed subset  of $\N ^ \N$ is itself a Polish space with the pointwise topology. If $S$ is a (topologically) closed submonoid of $\N ^ \N$, then $S$ is the endomorphism monoid of some relational structure on $\N$; and every such endomorphism monoid is a closed submonoid of $\N ^ \N$; see~\cite[Theorem 5.8]{Cameron1999aa} or~\cite[Proposition 6.1]{Cameron2006aa}.
As such every endomorphism monoid of a countable relational structure is a Polish monoid with respect to the pointwise topology; similarly, every automorphism group of a countable relational structure is a Polish group in this way.


In this paper we present general techniques for characterising minimal and maximal semigroup topologies on the endomorphism monoid $\End(\mathbb{A})$ of a  countable relational structure $\mathbb{A}$. As an application we show that the pointwise topology is the unique Polish semigroup topology on the endomorphism monoids of several well-known relational structures; see Corollary~\ref{apex}.

 In Section~\ref{section-fraisse},  we prove several results about minimal semigroup topologies.
 The main result in Section~\ref{section-fraisse} is Theorem~\ref{arsfacere}, which is stated below. The theorem applies to $\omega$-categorical homogeneous relational structures with no algebraicity that satisfy an additional, somewhat technical, property named arsfacere (see Definition~\ref{de-arsfacere}). Examples of such arsfacere relational structures include the random graph, the random directed graph, the random tournament, and several further well-known relational structures (see Theorem~\ref{cor-zariski-endomorphisms}).
Another notion that is central in Section~\ref{section-fraisse} is that of the Zariski topology on a monoid. 
 The \textit{Zariski} topology on a monoid $S$ is the topology with subbasis consisting
of
\[
  \set{s\in S}{(s)\phi_1 \not= (s)\phi_2}
\]
where $\phi_1, \phi_2: S\to S$ are any functions such that $(s)\phi_1 = t_1 s
  t_2 s \cdots t_{k - 1} s t_{k} $, $k\geq 1$ for every $s\in S$ and for some fixed
$t_1,
  \ldots, t_{k}\in S$, and $\phi_2$ is defined analogously for some fixed
$u_1, \ldots, u_{l} \in S$. This notion is analogous to the notion of the Zariski topology on a group. The Zariski topology on any monoid $S$ is $T_1$ and contained in every Hausdorff semigroup topology for $S$; see \cite[Propositions 2.1 and 2.2]{Elliott2019aa}. 

\begin{restatable}{ltheorem}{arsfacere}
\label{arsfacere}
 If $\mathbb{A}$ is a countable $\omega$-categorical homogeneous arsfacere relational structure with no algebraicity, then the Zariski topology and the pointwise topology coincide on every monoid $S$ such that $\Emb(\mathbb{A}) \leq S\leq \End(\mathbb{A})$. 
 \end{restatable}

In Section~\ref{section-maximal}, we consider maximal semigroup topologies on the endomorphism monoid $\End(\mathbb{X})$ of a homogeneous relational structure $\mathbb{X}$. This is achieved via property \textbf{X}, which is defined as follows. 
If $S$ is a topological semigroup and $A$ is a subset of $S$, then we say
that $S$ satisfies \textit{property \textbf{X} with respect to $A$} if the
following holds:
\begin{quote}
  for every $s\in S$ there exists $f_s, g_s\in S$ and $t_s\in A$ such that
  $s = f_s t_s g_s$ and for every neighbourhood $B$ of $t_s$ the set $f_s
    (B\cap
    A)g_s$ is a neighbourhood of $s$.
\end{quote}
It is shown in \cite[Theorem 3.1]{Elliott2019aa} that if a Polish semigroup $S$ has property \textbf{X} with respect to a Polish subgroup $G$, then the topology on $S$ is maximal among the Polish semigroup topologies on $S$.
We begin by establishing a sufficient condition, based on the existence of certain endomorphisms, for $\End(\mathbb{X})$ equipped with the pointwise topology to have property \textbf{X} with respect to the automorphism group $\Aut(\mathbb{X})$ of $\mathbb{X}$. 
We then proceed to isolate certain model-theoretic properties of $\mathbb X$ that imply this condition, and in particular introduce the \emph{strong amalgamation property with homomorphism gluing (SAHG)} in Definition~\ref{definition-homo-glue}. 

This property is  a strengthening of the classical strong amalgamation property
and can also be thought of as an ``almost free amalgamation property''.  
The main theorem proved in Section~\ref{section-maximal} is the following:
\begin{restatable}{ltheorem}{theomodeltheorypropertyx}
\label{theomodeltheorypropertyx}
Let \(\mathbb{X}\) be an $\omega$-categorical   relational structure which is homogeneous and
  homomorphism-homogeneous such that the age of \(\mathbb{X}\) has the strong
  amalgamation property with homomorphism gluing. Then \(\End(\mathbb{X})\)
  equipped with the pointwise topology has property \textbf{X} with respect to
  \(\Aut(\mathbb{X})\).
\ignore{
Let $\sigma$ be a relational signature and
  let \(\mathbb{X}\) be a \(\sigma\)-structure which is homogeneous and
  homomorphism-homogeneous such that the age of \(\mathbb{X}\) has the strong
  amalgamation property with homomorphism gluing. Then \(\End(\mathbb{X})\)
  equipped with the pointwise topology has property \textbf{X} with respect to
  \(\Aut(\mathbb{X})\).
}
\end{restatable}


In Section~\ref{section-apex}, we combine the results of the earlier sections to prove Corollary~\ref{apex}.
Defying the consensus of most of the literature in model theory, we do not require the edge relation of a directed graph to be antisymmetric (but do exclude loops), and remark that Proposition~\ref{prop:non-unique-top} can be used to prove that the statement does not hold for the  random directed graph without undirected edges.

\begin{restatable}{lcorollary}{apex}
\label{apex}
  The pointwise topology is the unique Polish semigroup topology on the endomorphism monoids of the following structures:
  \begin{enumerate}[\rm (i)]
      \item the random graph;
      \item the random directed graph;
      \item the random strict partial order;
      \item the graph \(\omega \mathbb{K}_n\) for any  $n\in \N\setminus\{0\}$;
      \item the graph \(n \mathbb{K}_\omega\) with loops for any $n\in (\N\setminus\{0\})\cup \{\omega\}$, i.e., the random equivalence relations with \(n\) countably infinite equivalence classes;
  \end{enumerate}
  as well as (i)-(iv) with all the loops included.
\end{restatable}

We will observe in Proposition \ref{prop-infinitely-many-wreath-product} that unlike its finitary counterparts, the endomorphism monoid of $\omega \mathbb{K}_{\omega}$ has infinitely many Polish semigroup topologies.

We then show, in Theorem~\ref{thm-final},  that the endomorphism monoids of several structures have the property that every homomorphism into a second countable topological semigroup is continuous; this property is called \textit{automatic continuity (with respect to the class of second countable topological semigroups)}. This is achieved by lifting the same property from the corresponding automorphism groups using property $\mathbf{X}$. 
 Such structures include: 
 the random graph; the random directed graph; and any
 random equivalence relation with infinite equivalence classes.
Automatic continuity for semigroups, groups, and clones, has been the subject of intensive research in recent years; see
\cite{Barbina2007,Bodirsky2014,Evans1990,Herwig1998,Lascar1991,Paolini2019,Paolini2020,  Rubin1994,Truss1989}. 
 
In Section \ref{section-clones}, we show how to extend the results about endomorphism monoids to clones of polymorphisms on the same structures.


\section{Preliminaries}

In this section, we introduce the notions and terminology required in later sections that are related to structures
and model theory.

\subsection{Functions} If $f: X \to Y$ is a partial function, then $f$ is a subset of $X\times Y$. 
The \textit{image of $f$} is the set $\im(f) = \set{y\in Y}{(x, y) \in f}$ and the 
\textit{domain of $f$} is $\dom(f) = \set{x\in X}{(x, y) \in f}$. If $f: X\to Y$ is a partial function, and $Z$ is a subset of $X$, then the \textit{restriction}  $f{\restriction_Z}$ of $f$ 
is just $f \cap (Z\times Y)$; and we say that $f$ is an \textit{extension} of $f{\restriction_Z}$. If $f, g: X \to Y$ are partial functions and $f\subseteq g$, then $f$ is the restriction of $g$ to $\dom(f)$ and $g$ is an extension of $f$.
If $f: X \to Y$ and $g: Z \to T$ are partial functions, then $f \cup g \subseteq (X\cup Z) \times (Y \cup T)$ is a partial function if and only if $f{\restriction_{\dom(f) \cap \dom(g)}} = g{\restriction_{\dom(f) \cap \dom(g)}}$. In particular, if $X \cap Z=\varnothing$, then $f\cup g$ is always a partial function.

The \emph{wreath product} of a semigroup $S$ with the monoid $N^N$ of all functions from a set $N$ to itself is the semigroup 
  \[
  S \wr N ^ N =\set{((f_i)_{i\in N}, g)}{f_i \in S \text{ for all } i\in N, g \in N ^ N}
  \]
  with multiplication
  \[
  ((f_i)_{i\in N}, g)((h_i)_{i\in N}, k)
  = 
  ((f_{i}h_{(i)g})_{i\in N}, gk).
  \]
 If $S$ is a topological semigroup, then the topology on $S \wr N ^ N$ is the product topology on \(S^N\times N^N\) where $N ^N$ has the pointwise topology (which is discrete in the case that $N$ is finite). This is a semigroup topology on the wreath product. If $S$ is Polish and $N$ is countable, then $S\wr N ^ N$ is Polish  also.

\subsection{Structures}
A \textit{signature} is a collection of names for constants, relations, and operations together with an associated finite arity for each relation and operation name. A \textit{structure}, or $\sigma$-structure, $\mathbb{X}$ in the signature $\sigma$ is a set $X$ together with constants, relations, and operations on $X$ of arity corresponding to those associated with each constant, relation, and operation name in $\sigma$.
We will use
blackboard letters such as $\mathbb{X}$, to denote structures, with the
exceptions of $\N$, $\Q$, and $\R$. The same letter in plain font will then
denote the domain of the structure, for example $X$ is the domain set of the
structure $\mathbb{X}$.
If $A$ is a symbol in the signature of $\mathbb{X}$, we will denote the corresponding constant, relation, or operation of $\mathbb{X}$ by $A^{\mathbb{X}}$. A \textit{relational structure} is a structure without operations and constants.
For example, the rational numbers with the usual order relation $\leq$ is a relational structure, and a monoid can be seen as a non-relational structure.

If $\mathbb{X}$ is a structure and $Y \subseteq X$, then the \emph{substructure generated} by $Y$ is the least substructure of $\mathbb{X}$ containing $Y$, all of the constants of $\mathbb{X}$, and closed under all of the operations in the signature of $\mathbb{X}$.
A structure $\mathbb{X}$ is \textit{finitely generated} if there is a finite $F \subseteq X$ such that the substructure of $\mathbb{X}$ generated by $F$ is $\mathbb{X}$ itself.

If $n\in \N$ and $n\geq 1$, then the \textit{$n$-th power} of a structure $\mathbb{X}$ with signature $\sigma$, denoted by $\mathbb{X}^n$, is a structure on $X^n$ defined as follows: 
\begin{enumerate}
    \item 
    $c^{\mathbb{X}^n} = (c^\mathbb{X}, \ldots, c^\mathbb{X}) \in X^n$ for every constant name $c$ in $\sigma$;
    \item 
    $((x^1_1, \ldots, x^1_{n }), \ldots,  (x^{m }_1, \ldots, x^{m}_{n })) \in R^{\mathbb{X}^n}$ if $(x^1_i, \ldots, x^{m }_i) \in R^\mathbb{X}$ for all $i \in \{1, \ldots, n \}$, all $m$-ary relation names  $R$ in $\sigma$, and all $m\in \N$;
    \item 
    $((x^1_1, \ldots, x^1_{n}), \ldots,  (x^{m}_1, \ldots, x^{m }_{n}))f^{\mathbb{X}^n} = ( (x^1_1, \ldots, x^{m  }_1)f^{\mathbb{X}}, \ldots, (x^1_{n}, \ldots, x^{m}_{n })f^{\mathbb{X}}) \in X^n$ 
    for every $m$-ary operation
    name $f$ in $\sigma$ and for every $m\in \N$.
\end{enumerate} 

Let $\mathbb{X}$ and $\mathbb{Y}$ be structures in some signature $\sigma$. Then a map $\psi: \mathbb{X} \to \mathbb{Y}$ is a \textit{homomorphism} if the following hold: 
\begin{enumerate}
\item $(c^\mathbb{X})\psi = c^\mathbb{Y}$ for every constant name $c$ in $\sigma$;
\item for every \(n \in \N\) and every $n$-ary relation name $R$ in $\sigma$: if $\x \in R^\mathbb{X}$ for some $\x \in X^n$, then $(\x)\psi \in  R^\mathbb{Y}$;
\item for every \(n \in \N\) and every $n$-ary operation name $f$ in $\sigma$:  $((\x)f^\mathbb{X})\psi = ((\x)\psi)f^\mathbb{Y}$ for every $\x \in X^n$.
\end{enumerate}
The homomorphism $\psi$ is an \textit{embedding} if it is injective and $\x \in R^\mathbb{X}$  if and only if $(\x)\psi \in  R^\mathbb{Y}$ for every  $\x \in X^n$ and every $n$-ary relation symbol where $n\in \N$. 
A surjective embedding is an \textit{isomorphism}. A homomorphism $\psi : \mathbb{X} \to \mathbb{X}$ is an \textit{endomorphism of $\mathbb{X}$}; if $\psi$ is an embedding, then it is a \emph{self-embedding of $\mathbb{X}$; and if $\psi$ is an isomorphism, then it is an \emph{automorphism of $\mathbb{X}$}.}

Since the composition of endomorphisms of a relational structure $\mathbb{X}$ is again an endomorphism, the collection $\End(\mathbb{X})$, of all endomorphisms of  $\mathbb{X}$, is a monoid. Similarly, the collections  $\Emb(\mathbb{X})$ of all self-embeddings and $\Aut(\mathbb{X})$ of all automorphisms are monoids; the latter is in fact a group. If $X$ is a set, then we denote by $\Inj(X)$ the monoid of injective functions from $X$ to $X$.
It is straightforward to verify
that $\Emb(\mathbb{X})$ is also a closed submonoid of each of the monoids $\Inj(X)$, $\End(\mathbb{X})$, and $X ^ X$ with the pointwise topology. 

A \textit{sentence} is a 
first-order formula without free variables and a \textit{theory} 
is a set of sentences in some fixed signature \(\sigma\). 
A \textit{model} \(\mathbb{M}\) of a theory \(T\) is a 
\(\sigma\)-structure such that each sentence in \(T\) is true in 
\(\mathbb{M}\). A theory is \textit{\(\omega\)-categorical} if it 
has a unique countably infinite model up to isomorphism. A 
structure is \(\omega\)-categorical if its theory is. By 
Ryll-Nardzewski's theorem (see~\cite{Hodges1997aa})  
a countable structure $\mathbb{A}$ is  
\textit{\(\omega\)-categorical} if and only if for every $k\in 
\N$, the number of orbits of the action of $\Aut(\mathbb{A})$ on the $k$-element subsets of $A$ is finite.

If $\mathbb{A}$ is a structure, and $F$ is a subset of $A$, then 
the \textit{pointwise stabiliser} of $F$ in $\Aut(\mathbb{A})$ is defined to be
\[ \Aut(\mathbb{A})_F := \{f \in \Aut(\mathbb{A}) \colon (x)f = x \text{ for all } x\in F\}.\]
An \(\omega\)-categorical structure \(\mathbb{A}\) has \emph{no algebraicity} if 
every orbit on $A\setminus F$ of the pointwise stabiliser $\Aut(\mathbb{A})_F$ is infinite for every finite subset $F$ of $A$.


\subsection{Fra\"{i}ss\'{e} theory}
We define the following properties for a class $\mathcal{K}$ of finitely generated structures in a fixed countable signature \(\sigma\).
\begin{enumerate}
  \item[\textbf{HP:}]
        \textit{hereditary property}:
        if $\mathbb{A}$ is a substructure of $\mathbb{B}$ and $\mathbb{B}\in \mathcal{K}$, then
        $\mathbb{A}\in\mathcal{K}$.

  \item[\textbf{JEP:}]
        \textit{joint embedding property}:
        if $\mathbb{A}, \mathbb{B}\in \mathcal{K}$, then there exists $\mathbb{C}\in \mathcal{K}$ such that both $\mathbb{A}$ and $\mathbb{B}$ embed as substructures of $\mathbb{C}$.

  \item[\textbf{AP:}]
        \textit{amalgamation property}: if $\mathbb{A}, \mathbb{B}_1, \mathbb{B}_2\in \mathcal{K}$, 
         $f_1: \mathbb{A} \to \mathbb{B}_1$ and $f_2: \mathbb{A} \to \mathbb{B}_2$ are embeddings, then there exists
        $\mathbb{C}\in \mathcal{K}$ and embeddings $g_1: \mathbb{B}_1 \to \mathbb{C}$ and $g_2: \mathbb{B}_2 \to \mathbb{C}$ such
        that $f_1\circ g_1 = f_2 \circ g_2$.
        
        \item[\textbf{SAP:}]
        \textit{strong amalgamation property} is the same as \textbf{AP} with the additional property that \(\im(g_1) \cap \im(g_2) = (A) f_1 g_1 = (A)f_2 g_2\).
        
        \item[\textbf{FAP:}]
        if \(\sigma\) is a relational signature, then 
        the \textit{free amalgamation property} is the same as \textbf{SAP} with the additional property that if 
        \( (x_1, \ldots, x_n) \in R^{\mathbb{C}}\) for some \(R \in \sigma\), then either \(x_1, 
        \ldots, x_n \in \im(g_1)\) or \(x_1, \ldots, x_n \in \im(g_2)\).
\end{enumerate}

A class $\mathcal{K}$ of finitely generated structures in a fixed countable signature is called a \textit{Fra\"iss\'e class} if $\mathcal{K}$ is closed under taking isomorphims, has countably many isomorphism classes, and satisfies \textbf{HP}, \textbf{JEP}, and \textbf{AP}.
Some examples of Fra\"iss\'e classes, among many others, include: finite graphs, finite partial orders, finite linear orders, non-trivial finite  Boolean algebras, and finite metric spaces with rational distances.

The \emph{age} of a structure \(\mathbb{A}\), denoted by \(\Age(\mathbb{A})\),  is the class of all finitely generated structures which embed into \(\mathbb{A}\).
A structure $\mathbb{X}$ is \textit{homogeneous} if every isomorphism between finitely generated substructures of $\mathbb{X}$ can be extended to an automorphism of $\mathbb{X}$.
Associated to every Fra\"iss\'e class $\mathcal{K}$ is a \textit{Fra\"iss\'e limit} $\mathbb{K}$ which is the unique, up to isomorphism, countable structure in the same signature as the structures in $\mathcal{K}$ such that $\mathbb{K}$ is homogeneous and \(\Age(\mathbb{K})\)
equals
$\mathcal{K}$; see~\cite{Fraisse2000aa} and~\cite[Theorem~6.1.2]{Hodges1997aa} for more
details. 
Examples of Fra\"iss\'e limits include the countably infinite random graph~\cite{Rado1964aa} and its directed counterpart, the linear order of rational numbers $\Q$, the random (reflexive) partial order and its strict counterpart, 
the countably infinite atomless Boolean algebra (see, for example, ~\cite[Theorem 10]{givant2008introduction}), and the rational Urysohn space (see, for example, ~\cite{melleray2008some}.
These examples are the Fra\"iss\'e limits of the classes of all: finite graphs, finite directed graphs, finite linear orders, finite partial orders and finite strict partial orders, non-trivial finite Boolean algebras, and finite metric spaces with rational distances, respectively.
For the purpose of this paper, \emph{a directed graph} is a binary relation which is anti-reflexive and not necessarily symmetric. 

The definition of \emph{homomorphism-homogeneity} is similar to that of homogeneity: if \(\mathbb{X}\) is a structure, then every homomorphism between finitely generated substructures of \(\mathbb{X}\) can be extended to an endomorphism, see~\cite{Cameron2006aa}.

Another structure we will consider in this paper is \(n \mathbb{A}\), that is the disjoint union of \(n\) copies of a relational structure \(\mathbb{A}\) where $n$ is any cardinal.  We are concerned in particular with \(n \mathbb{K}_m\), where \(\mathbb{K}_m\) is the complete graph on \(m\) vertices where $m\in \N\cup \{\omega\}$. 
Clearly, \(n \mathbb{K}_m\) is homogeneous for all \(n, m \leq \omega\).

If \(R\) is a binary relation on a set \(A\), then a \emph{loop} is any pair \( (x, x) \in R\).  The class of finite graphs where every vertex has a loop is a Fra\"{i}ss\'{e} class, and so, in the same way as above, we obtain the random graph with loops. Similarly, we may construct the following structures as Fra\"{i}ss\'{e} limits -- the random reflexive partial order \( (P,  \leq)\), the random directed graph with loops, and \(n \mathbb{K}_m\) with loops.


\section{Minimal topologies}\label{section-fraisse}

In this section, we prove Theorem~\ref{arsfacere} and two further results (Theorems~\ref{theorem-the-first-theorem} and~\ref{cor-zariski-endomorphisms}) establishing minimal semigroup topologies on certain monoids of endomorphisms of relational structures. 
In Theorem~\ref{theorem-the-first-theorem}, we show that the pointwise topology is the minimal \(T_1\) topology that is semitopological (defined below) for the endomorphism monoids of a number of well-known homogeneous relational structures with loops. 
It will follow that the Zariski topology on each of these endomorphism monoids is the pointwise topology. 
In Theorem~\ref{cor-zariski-endomorphisms}, we apply Theorem~\ref{arsfacere} to show that the Zariski topology and the pointwise topology coincide for the endomorphism monoids of  several homogeneous structures without loops. As a consequence, every one of these structures has the property that every Hausdorff semigroup topology on its endomorphism monoid contains the pointwise topology. 

It is routine to show that there are relational structures $\mathbb{A}$ where the Zariski topology on $\End(\mathbb{A})$ is strictly contained in the pointwise topology. For example, the monoid $M$ generated by the constant transformations form a countable closed submonoid of $\N ^ \N$ and $M$ is a right zero semigroup (a semigroup satisfying $xy=y$ for all elements $x,y$) with identity adjoined. It follows that $M$ is the endomorphism monoid of some countable relational structure $\mathbb{A}$. The pointwise topology induces the discrete topology on $M$. 
Since the Zariski topology is contained in every Hausdorff semigroup topology, it is in particular  contained in the topology generated by the cofinite topology on the constant transformations and the singleton consisting of the identity. Hence, the Zariski topology is not discrete. 

\begin{question}
  Is there an \(\omega\)-categorical relational structure \(\mathbb{A}\) such that the topology of pointwise convergence on \(\End(\mathbb{A})\) is strictly finer than the Zariski topology? 
\end{question}

If $S$ is a semigroup and $x\in S$, then we define $\lambda_x: S\to S$ and $\rho_x: S \to S$  by 
$(y)\lambda_x = xy$ and $(y)\rho_x = yx$ for all $y\in S$. 
A semigroup $S$ with a topology $\T$ on $S$ is called \textit{semitopological}
if $\lambda_x$ and $\rho_x$ are continuous for every $x\in S$. Every topological semigroup is semitopological, but the converse is not true.

To prove the first of the main theorems in this section, we require the following lemma from~\cite{Elliott2019aa}.

\begin{lem}[cf. Lemma 5.1 in \cite{Elliott2019aa}]\label{lem-luke-somenumber}
    Let $X$ be an infinite set, and let $S$ be a subsemigroup of $X ^ X$ such that
  $S$ contains all of the constant  transformations, and for every $x\in X$
  there exists $f_{x}\in S$ such that $(x) f_{x} ^ {-1} = \{x\}$ and $(X)f_{x}$
  is finite.  If $\T$ is a topology which is semitopological for
  $S$, then the following are equivalent:
  \begin{enumerate}[\rm (i)]
    \item $\T$ is Hausdorff;
    \item $\T$ is $T_1$;
    \item $\{f\in S: (y)f = z\}$ is open in
          $\T$ for all $y, z \in X$;
    \item $\{f\in S: (y)f = z \}$ is closed in
          $\T$ for all $y, z \in X$.
  \end{enumerate}
\end{lem}

In the next result we use Lemma~\ref{lem-luke-somenumber} to show that the pointwise topology is in some sense minimal for the endomorphism monoids of several Fra\"iss\'e limits. 

\begin{theorem}\label{theorem-the-first-theorem}
  The pointwise topology is contained in every \(T_1\) topology that is semitopological for the endomorphism monoid of the following structures:
  \begin{enumerate}[\rm (i)]
    \item \((\mathbb{Q}, \leq)\);

    \item the random reflexive partial order;
    
    \item the random graph with loops;
    
    \item the random directed graph with loops;
    
    \item \(n \mathbb{K}_m\) with loops for every \(1\leq n, m \leq \omega\) such that either \(n = \omega\) or \(m = \omega\).
  \end{enumerate} 
\end{theorem}
\begin{proof}
  Each case of the theorem will follow from Lemma~\ref{lem-luke-somenumber}.
  First,  note that if \(\mathbb{X}\) is one of the five structures appearing in the theorem, then every constant map is an endomorphism of \(\mathbb{X}\). Hence it remains to show that for every \(x \in X\) there exists \(f_x \in \End(\mathbb{X})\) such that \((x)f_x ^ {-1} = \{x\}\) and
  \((X)f_x\) is finite.
  \medskip
  
  \noindent\textbf{(i).}
  If \(x \in \mathbb{Q}\), then $f_x\in\EndQ$ might be defined by
  \[
    (y)f_x =
    \begin{cases}
      x - 1 & \text{if } y < x  \\
      x     & \text{if } y = x  \\
      x + 1 & \text{if } y > x. \\
    \end{cases}
  \]
  
  \noindent\textbf{(ii).}
  Let $(P, \leq)$ be the random reflexive partial order. There exist injective homomorphisms $g\colon (P, \leq)\to (\mathbb Q, \leq)$ and $h\colon (\mathbb Q, \leq) \to (P, \leq)$; given $x\in P$, we may moreover assume that $(x)gh=x$ by the transitivity of the automorphism group of $(P, \leq)$. Then $gf_{g(x)}h$, where $f_{g(x)}$ is as in (i), is the required  endomorphism.
  
  \noindent\textbf{(iii).} Let \(x\) be a vertex of the random graph with loops, and let \(y\) be any other vertex adjacent to \(x\). Then \(f_x\) defined by \( (x)f_x = x\) and \( (z)f_x = y\) for every \(z \neq x\) is the required endomorphism.
  \medskip
  
  \noindent\textbf{(iv).} The functions \(f_x\) can be defined in the same way as in (iii) where $y$ is any vertex such that $(y,x)$ and $(x,y)$ are edges.
  \medskip
  
  \noindent\textbf{(v).} If \(m > 1\), the functions \(f_x\) can be defined in the same way as in (iii). If \(m = 1\), then \(n = \omega\), and so \(\End(\omega \mathbb{K}_1) = X ^ X\), where $X$ is the set of vertices of $\omega \mathbb{K}_1$, contains every transformation on the domain of \(\omega \mathbb{K}_1\), and hence contains the required functions $f_x$ for every $x\in X$.
\end{proof}

Next, we define a condition on relational structures $\mathbb{A}$ that will permit us to establish a sufficient condition for the Zariski topology and pointwise topology on $\End(\mathbb{A})$ to coincide. 
\begin{de}\label{de-arsfacere}
 If $\mathbb{A}$ is a relational structure, then we say that $\mathbb{A}$ is \textit{arsfacere} if
 there is \(n \in \mathbb{N}\) and an \(n\)-ary relation symbol \(E\) in the signature of \(\mathbb{A}\) such that the following conditions hold:
  \begin{enumerate}[\rm (i)]
    \item \( (x, \ldots, x) \notin E^{\mathbb{A}}\) for all \(x \in A\);

    \item if \(x \in A\) and \(F_1, \ldots, F_n \subseteq A\)  are finite such that \(F_i \cap F_j = \{x\}\) for all $i, j\in \{1, \ldots, n\}$,  \(i \neq j\), then there is an injection \(f \colon \bigcup_{i = 1}^n F_i \to A\) such that \(f{\restriction_{F_i}}\) is an isomorphism between induced substructures of \(\mathbb{A}\) for each \(i \in \{1, \ldots, n\}\) and for every \((x_1, \ldots, x_n) \in (F_1 \setminus \{x\}) \times \ldots \times (F_n \setminus \{x\})\) some permutation of \( ((x_1)f, \ldots, (x_n)f)\) is in \(E^{\mathbb{A}}\).
  \end{enumerate}
 \end{de}

Note that the second assumption in Definition~\ref{de-arsfacere}  holds trivially for structures which have a relation \(E\) such that some permutation of every tuple of distinct elements is in \(E\), for example \((\mathbb{Q}, <)\) or the random tournament. It is routine to verify that several further structures, such as the random graph, are arsfacere.
 
 
 The purpose of the remainder of this section is to establish the proof of Theorem~\ref{arsfacere}.

\arsfacere*
 
We prove Theorem~\ref{arsfacere} in a sequence of lemmas. 

\begin{lem}[cf. Lemma 5.3 in \cite{Elliott2019aa}]\thlabel{lem-zariski-pointwise}
  Let \(X\) be an infinite set and let \(S\) be a subsemigroup of \(X^X\) such
  that for every \(a \in X\) there exist \( \alpha, \beta, \gamma_{1}, \ldots, \gamma_{n} \in S \)
  for some $n\in \N$ such that the following hold:
  \begin{enumerate}[\rm (i)]
    \item  \((x)\alpha = (x)\beta\) if and only if \(x \neq a\);

    \item
          \(a \in \im(\gamma_{i})\) for all \(i\in \{1, \ldots, n\}\);

    \item
          for every \(s \in S\) and every \(x \in X\setminus \{(a)s\}\) there
          is \(i \in \{1, \ldots, n\}\) so that \( \im(\gamma_{i}) \cap (x)s^{-1} =
          \varnothing\).
  \end{enumerate}
  Then the Zariski topology of \(S\) is the pointwise topology.
\end{lem}      
The following two lemmas provide sufficient conditions for the assumptions of
\thref{lem-zariski-pointwise} 
to be satisfied by the monoids of endomorphisms, and embeddings, of a relational structure. 
This will allow us to prove  Theorem~\ref{arsfacere} and, in Theorem~\ref{cor-zariski-endomorphisms}, to prove that the Zariski topology coincides with the pointwise topology on several examples of endomorphism and embedding monoids of relational structures. 


\begin{lem}\thlabel{lem-zariski-functions1}
  Let \(\mathbb{A}\) be an $\omega$-categorical relational structure  with no algebraicity. Then for every \(a \in A\) there are \( \alpha, \beta \in \Emb(\mathbb{A}) \) such that \((x)\alpha = (x)\beta\) if and only if \(x \neq a\).
\end{lem}
\begin{proof}
  Let \(a \in A\) be fixed and let \( F_0 = \{a\} \subseteq F_1 \subseteq \cdots\) be finite sets such that \(\bigcup_{n \in \N} F_n = A\).
  Suppose that \(n \in \N\) and let \(\alpha_n\) be the identity map. Since \(\mathbb{A}\) has no algebraicity there exists 
  \(\beta_n \in \Aut(\mathbb{A})_{F_n \setminus \{a\}}\) such that \((a)\beta_n \neq a\). Then for
  all \(n \in \N\)
  \begin{equation}\label{equation-well-named}
      (x)\alpha_n = (x)\beta_n \ \text{for all} \ x \in F_n \setminus \{a\} \ \text{and} \ (a)\alpha_n \neq (a)\beta_n.
  \end{equation}
  
  Since \(\mathbb{A}\) is \(\omega\)-categorical, there is \(i_0\) such that for infinitely many \(k \geq i_0\) there is \(u_k \in \Aut(\mathbb{A})\) such that \( \alpha_k u_k\) and \(\beta_k u_k\) agree with \(\alpha_{i_0}\) and \(\beta_{i_0}\) on \(F_0\). Next, replace \(\alpha_0\) and \(\beta_0\) by \(\alpha_{i_0}\) and \(\beta_{i_0}\), replace \(\alpha_k\) and \(\beta_k\) by \(\alpha_k u_k\) and \(\beta_k u_k\) for every \(k\) as above, and remove the remaining \(\alpha_m\) and \(\beta_m\).
  Note that~\eqref{equation-well-named} still holds for the new sequences.
  Subsequently repeat this process for \(n \geq 1\) each time choosing \(i_n \geq n\). In the end, we obtain convergent sequences \(\alpha_0, \alpha_1, \ldots\) and \(\beta_0, \beta_1, \ldots\) of automorphisms such that \eqref{equation-well-named} still holds. If \(\alpha, \beta \in \overline{\Aut(\mathbb{A})}\) are limits of these sequences, it follows from~\eqref{equation-well-named} that \( (x)\alpha = (x)\beta\) if and only if \(x \neq a\). Finally, \(\overline{\Aut(\mathbb{A})} \subseteq \Emb(\mathbb{A})\), completing the proof.
\end{proof}

We note that it is also possible to prove Lemma~\ref{lem-zariski-functions1} using~\cite[Proposition 6]{Bodirsky2014ab}, which implies that a certain quotient of $\Emb(\mathbb{A})\times \Emb(\mathbb{A})$ is compact.

We require the following simple observation about automorphism groups of structures with no algebraicity.
The proof is omitted. 

  \begin{lem}\label{claim-pointwise-stab}
    Let $\mathbb{A}$ be a relational structure with no algebraicity,
    let \(a \in A\), and let \(F, H \subseteq A\) be finite subsets so that \(a \in F \cap H\). Then there is \(\alpha \in \Aut(\mathbb{A})\) such that \( (F)\alpha \cap H = \{a\}\).
  \end{lem}

  
  The final lemma in our sequence is the following.
  
\begin{lem}\thlabel{lem-zariski-functions2}
  Let \(\mathbb{A}\) be a countable homogeneous  relational structure with no algebraicity which is arsfacere as witnessed by an $n$-ary relation symbol $E$.  
  Then for every \(a \in A\) there are \(\gamma_{1}, \ldots, \gamma_{n} \in \Emb(\mathbb{A})\) such that \(\im(\gamma_{i}) \cap \im(\gamma_{j}) = \{a\}\) for all $i, j\in \{1, \ldots, n\}$, \(i \neq j\),  and such that for every \(f \in \End(\mathbb{A})\) and every \(b \in A\setminus \{(a)f\}\) there is \(i \in \{1, \ldots, n\}\) so that \( \im(\gamma_{i}) \cap (b)f^{-1} = \varnothing\).
\end{lem}

\begin{proof}
  We proceed in a similar way to \thref{lem-zariski-functions1}. Let \(a \in A\) be fixed.
   First, we will show that if \(F \subseteq A\) is finite and \(a \in F\), then there are \(\alpha_1, \ldots, \alpha_n \in \Aut(\mathbb{A})_a\) such that \((F)\alpha_i \cap (F)\alpha_j = \{a\}\) for all $i, j\in \{1, \ldots, n\}$, \(i \neq j\), and
  for every \((x_1, \ldots, x_n) \in ((F)\alpha_1 \setminus \{a\}) \times \cdots \times  ((F)\alpha_n \setminus \{a\})\) there is some permutation of $(x_1, \ldots, x_n)$ in $E ^ {\mathbb{A}}$.
  We denote this property of \(\alpha_1, \ldots, \alpha_n\) by \(\mathcal{P}(F)\).
 
  Fix a finite \(F \subseteq A\) and let \(\beta_1 \in \Aut(\mathbb{A})\) be the identity map.
  By repeated application of Lemma~\ref{claim-pointwise-stab}, we obtain \(\beta_2, \ldots, \beta_n \in \Aut(\mathbb{A})_a\) such that  \((F)\beta_i \cap (F)\beta_j = \{a\}\) for all \(i, j \in \{1, \ldots, n\}\), \(i \neq j\). Since $\mathbb{A}$ is arsfacere, there is an injection \(f \colon \bigcup_{i = 1}^n (F)\beta_i \to A\) such that \(f{\restriction_{(F)\beta_i}}\) is an isomorphism between induced substructures of \(\mathbb{A}\) for each \(i \in \{1, \ldots, n\}\), and such that for every \((x_1, \ldots, x_n) \in ( (F)\beta_1 \setminus \{a\}) \times \cdots \times ((F)\beta_n \setminus \{a\})\)
  some permutation of \( ((x_1)f, \ldots, (x_n)f)\) belongs to \(E^{\mathbb{A}}\).
  By homogeneity, 
  we may assume that $f{\restriction_{(F)\beta_1}}$ is the identity, and, in particular, $f$ fixes $a$.
  Then \( \beta_1 f{\restriction_{F}}, \ldots, \beta_n f{\restriction_{F}}\) are isomorphisms between induced substructures of \(\mathbb{A}\), and since \(\mathbb{A}\) is homogeneous, they can be extended to automorphisms \(\alpha_1, \ldots, \alpha_n\) respectively. It is routine to verify that \(\alpha_1, \ldots, \alpha_n\) satisfy \(\mathcal{P}(F)\).
  
  Let \( F_0 = \{a\} \subseteq F_1 \subseteq \cdots\) be finite sets such that \(\bigcup_{n \in \N} F_n = A\). 
  By the first paragraph of the proof, for every \(i \in \N\) there are \(\gamma_{i, 1}, \ldots, \gamma_{i, n}\in \Aut(\mathbb{A})\) satisfying \(\mathcal{P}(F_i)\). 
  By a similar argument as in the proof of Lemma~\ref{lem-zariski-functions1}, by replacing each of the sequences $(\gamma_{i,j})_{i\in\mathbb N}$, where 
    $j\in \{1, \ldots, n\}$, by an appropriate subsequence, we may assume that these sequences converge. 
  We denote the limits of these subsequences by 
 \(\gamma_1, \ldots, \gamma_n \in \overline{\Aut(\mathbb{A})} \subseteq \Emb(\mathbb{A})\).
  It follows that \(\im(\gamma_i) \cap \im(\gamma_j) = \{a\}\) for all $i, j\in \{1, \ldots, n\}$,  \(i \neq j\), and for every \((x_1, \ldots, x_n) \in (\im(\gamma_1) \setminus \{a\}) \times \cdots \times  (\im(\gamma_n) \setminus \{a\})\), some permutation of \((x_1, \ldots, x_n)\) belongs to $E^{\mathbb{A}}$.

  Let \(f \in \End(\mathbb{A})\) and \(b \in A\setminus \{(a)f\}\). Suppose that for every \(i \in \{1, \dots, n\}\) there is \(x_i \in \im(\gamma_{i}) \cap (b)f^{-1}\).
  Observe that \( (x_i)f = b \neq (a)f\), and so \(x_i \in \im(\gamma_i) \setminus \{a\}\) for all \(i \in \{1, \dots, n\}\). It follows from above that some permutation of 
  \( (x_1, \ldots, x_n)\) belongs to \(E^{\mathbb{A}}\) and so \( (b, \dots, b)  \in E^{\mathbb{A}}\), contradicting the assumption of the lemma. Therefore there is \(i \in \{1, \dots, n\}\) such that \(\im(\gamma_{i}) \cap (b)f^{-1} = \varnothing\), as required.
\end{proof}

\begin{proof}[Proof of Theorem~\ref{arsfacere}]
Suppose that $S$ is any submonoid of $\End(\mathbb{A})$ containing $\Emb(\mathbb{A})$. 
Since $\mathbb{A}$ is an arsfacere $\omega$-homogeneous relational structure with no algebraicity, Lemma~\ref{lem-zariski-functions1} implies that Lemma~\ref{lem-zariski-pointwise}(i) holds, and Lemma~\ref{lem-zariski-functions2}
implies that Lemma~\ref{lem-zariski-pointwise}(ii) and (iii) hold too. Hence, by Lemma~\ref{lem-zariski-pointwise}, 
the Zariski topology and the pointwise topology coincide on $S$.
\end{proof}

\begin{theorem}\label{cor-zariski-endomorphisms}
  The Zariski topology of \(\End(\mathbb X)\) is the pointwise topology whenever \(\mathbb{X}\) is one of the following structures: 
    \begin{enumerate}[\rm (i)]
        \item the random graph;
        \item the random directed graph;
        \item \((\mathbb{Q}, <)\);
        \item the random tournament;
        \item the random strict partial order;
        \item \(n \mathbb{K}_\omega\) for every $n\in (\mathbb{N}\setminus\{0\})\cup \{\omega\}$;
        \item any structure $\mathbb Y$ with a first-order definition in any of the structures $\mathbb X$ in items (i)-(iv) above  which (like $\mathbb X$)   has no endomorphism sending a pair in the  relation of $\mathbb X$ onto a loop; this is in particular the case for all expansions of the structures $\mathbb X$ by first-order definable relations.
  \end{enumerate}
\end{theorem}
\begin{proof}
  We proceed by showing that Lemma~\ref{lem-zariski-pointwise} applies in each of the cases of the theorem. In particular, we will observe  that the structures in the statement are $\omega$-categorical with no algebraicity, and so there exist \(\alpha, \beta\in\Emb(\mathbb{X})\) satisfying 
  \thref{lem-zariski-pointwise}(i) by \thref{lem-zariski-functions1}.
  The remaining functions \(\gamma_{1}, \ldots, \gamma_{m}\) in the hypothesis of \thref{lem-zariski-pointwise} are obtained using \thref{lem-zariski-functions2} in cases (i)-(iv) (i.e., the structures in cases (i)-(iv) satisfy the hypothesis of Theorem~\ref{arsfacere}) and explicitly constructed in cases (v) and (vi). Case (vii) \ljy{will} 
  follow from the argument for the cases (i)-(iv).
  
   It is well-known  that the ages of \((\mathbb{Q}, <)\), the random strict partial order, the random graph, the random directed graph, the random tournament, and \(n \mathbb{K}_\omega\) have the strong amalgamation property. Moreover, each of these structures is a Fra\"{i}ss\'{e} limit, and so they are all homogeneous. Since the strong amalgamation property is equivalent to no algebraicity for homogeneous structures~\cite[(2.15)]{Cameron1990}, it follows that all these structures  have no algebraicity. 
  Furthermore, if \(\mathbb{Y}\) is first-order definable in \(\mathbb{X}\), then \(\Aut(\mathbb{X}) \leq\Aut(\mathbb{Y})\), and so if \(\mathbb{X}\) has no algebraicity, so does \(\mathbb{Y}\). Hence the structures in case~(vii) also have no algebraicity. 
  
  As discussed after Definition~\ref{de-arsfacere},  arsfacere holds trivially in the cases of \((\mathbb{Q}, <)\) and the random tournament, since in each of these structures some permutation of every tuple is related in the unique relation of that structure. Clearly, the random graph \(\mathbb{G} = (G, E^{\mathbb{G}})\) is arsfacere: part (i) of Definition~\ref{de-arsfacere} holds for all graph without loops, and for (ii) given $F_1,F_2\subseteq G$ with $F_1\cap F_2=\{x\}$ the required  injection $f$ exists since the graph obtained from the induced subgraph on $F_1\cup F_2$ by adding all edges between pairs in $(F_1\setminus\{x\})\times (F_2\setminus\{x\})$ is a subgraph of \(\mathbb{G}\). The argument for the random digraph is similar. 
  
  We consider case~(vii). Since $\mathbb Y$ is first-order definable in $\mathbb X$, every automorphism of $\mathbb X$ is an automorphism of $\mathbb Y$. Consequently the endomorphism monoid of $\mathbb Y$ contains the  functions $\gamma_1,\gamma_2$ constructed using arsfacere in Lemma~\ref{lem-zariski-functions2}, since these endomorphisms are obtained as limits of automorphisms. By the condition of the endomorphisms of $\mathbb Y$, not sending any pair in the  relation in $\mathbb X$ onto a loop, these functions also work as witnesses of Lemma~\ref{lem-zariski-pointwise}~(ii) and~(iii) for $\mathbb Y$.
  
  The structure \(n \mathbb{K}_\omega\) is not arsfacere, nor is the random strict partial order; in these cases, we verify directly that there exist endomorphisms satisfying \thref{lem-zariski-pointwise}(ii) and~(iii). Note that for arsfacere structures Lemma~\ref{lem-zariski-functions2} provides self-embeddings witnessing \thref{lem-zariski-pointwise}(ii) and~(iii), but only endomorphisms are needed. The endomorphisms we will  construct work for the same reason as the self-embeddings provided by  Lemma~\ref{lem-zariski-functions2} in the case of arsfacere structures. 
  Denote the domain of \(n \mathbb{K}_\omega\) by \(D\).
  Let \(a \in D\) be fixed, and let \(C_1, C_2 \subseteq D\) be infinite such that \(C_1 \cap C_2 = \{a\}\) and \(C_1 \cup C_2\) is contained in a copy of \(\mathbb{K}_\omega\). For \(i \in \{1, 2\}\), let \(\gamma_{i} \colon D \to C_i\) be any surjection. Then \(\gamma_{1}, \gamma_{2} \in \End(n \mathbb{K}_\omega)\) and \(a \in \im(\gamma_{i})\) for \(i \in \{1, 2\}\).
  Suppose that \(s \in \End(n \mathbb{K}_\omega)\) and let \(x \in D \setminus \{(a)s\}\). Then \(a \notin (x)s^{-1}\), and so if there are \(x_i \in \im(\gamma_{i}) \cap (x)s^{-1}\) for both \(i \in \{1,2\}\), then $x_1 \neq a$ and $x_2 \neq a$, and so \(x_1\) is adjacent to \(x_2\). Hence \( (x, x) = ( (x_1)s, (x_2)s)\) is an edge in \(n \mathbb{K}_\omega\), which is a contradiction. Therefore, 
  \(\im(\gamma_{i}) \cap (x)s^{-1} = \varnothing\) for some \(i \in \{1, 2\}\).
  
  We prove the same for the random strict partial order, which we denote by \((P, <)\). The argument is similar to the argument for $n\mathbb{K}_{\omega}$. Let \(a \in P\) be fixed. Pick any set $A\subseteq P$ with $a\in A$ which induces a dense linear order in \((P, <)\). Next pick  $A_1, A_2 \subseteq A$ such that $A_1\cap A_2=\{a\}$ and such that both sets induce a dense linear order without endpoints in \((P, <)\). Since the order relation of \((P, <)\) can be extended to a linear order, it follows that for each \(i \in \{1, 2\}\), we can find an (injective) endomorphism $\gamma_i$ of \((P, <)\) with \(\im(\gamma_{i}) = A_i\).  It can be shown in a similar way as above that \(\gamma_{1}\) and \(\gamma_{2}\) are the required endomorphisms satisfying \thref{lem-zariski-pointwise}(ii) and (iii).
\end{proof}

We conclude the section with yet another case when the pointwise topology is minimal.

\begin{prop}\label{prop-luke-9}
Let \(\mathbb{A}\) be a finite relational structure and let $\omega \mathbb{A}$ be the disjoint union of countably infinitely many copies of $\mathbb{A}$. Then every \(T_1\) topology semitopological for \(\End(\omega \mathbb{A})\) contains the pointwise topology.
\end{prop}
\begin{proof}
Suppose \(\End(\omega \mathbb{A})\) is semitopological with respect to some \(T_1\) topology. 
Let \(\mathbb{A}_0, \mathbb{A}_1, \ldots\) denote the disjoint copies of \(\mathbb{A}\) in \(\omega \mathbb{A}\), and let \(a, b\in \omega \mathbb{A}\) be arbitrary. 
It suffices to show that the set
\[U_{a, b}=\set{f\in \End(\omega\mathbb{A})}{(a)f=b}\]
is open.

Suppose that $i,j \in \N$ are such that $a\in A_{i}$ and $b\in A_{j}$, and let $k\in \N\setminus \{i, j\}$ be arbitrary. We pick $f\in \End(\omega \mathbb{A})$ such that $f|_{A_\ell}$ is an isomorphism from $A_\ell$ to $A_{i}$ for all $\ell\in \N\setminus\{i\}$ and $f|_{A_{i}}$ is the identity. Moreover, we pick $g\in \End(\omega\mathbb{A})$ such that $g$ is the identity on $A_{j}$ and $g|_{A_\ell}$ is an isomorphism from $A_\ell$ to $A_{k}$ for every $\ell\in \N \setminus\{j\}$.
%

Let \(\phi: \End(\omega\mathbb{A}) \to \End(\omega\mathbb{A})\) be the continuous map defined by \((h)\phi=fhg\).
By the construction of \(f\), it follows that for all \(t_1, t_2\in \im(\phi)\), 
\[t_1= t_2\quad \text{if and only if}\quad t_1{\restriction_{A_{i}}}= t_2{\restriction_{A_{i}}}.\]
By the construction of \(g\), \(\set{t{\restriction_{A_{i}}}}{t\in \im(\phi)}\) 
consists of all homomorphisms from $\mathbb{A}_{i}$ to $\mathbb{A}_{j} \cup \mathbb{A}_{k}$. In particular, this set is finite and so \(\phi\) has finite image. Since \(\End(\omega\mathbb{A})\) is \(T_1\), it follows that \(\im(\phi)\) is discrete and so the kernel classes of \(\phi\) are open. It follows that the set
\[V=\bigcup_{\set{t\in \im(\phi)}{(a)t= b}} (t)\phi^{-1}=\set{h \in \End(\omega \mathbb{A})}{(a)fhg=b}\]
is open. 
But, for any \(h \in \End(\omega \mathbb{A})\), \((a)h= b\) if and only if \((a)fhg=(a)((h)\phi)= b\), and so 
\(V=U_{a, b}\)
is open.
\end{proof}

\section{Property \textbf{X} and maximal topologies}
\label{section-maximal}
Recall from Section~\ref{section-intro} that if $S$ is a topological semigroup and $A$ is a subset of $S$, then we say
that $S$ satisfies \textit{property \textbf{X} with respect to $A$} if the
following holds:
\begin{quote}
  for every $s\in S$ there exists $f_s, g_s\in S$ and $t_s\in A$ such that
  $s = f_s t_s g_s$ and for every neighbourhood $B$ of $t_s$ the set $f_s
    (B\cap
    A)g_s$ is a neighbourhood of $s$.
\end{quote}
Property \textbf{X} is the crucial ingredient that we use to  determine maximal Polish
semigroup topologies, and to show automatic continuity.
The following theorem establishes the connection between property \textbf{X}, maximal Polish topologies, and automatic continuity. 
\begin{theorem}[cf. Theorem 3.1 in \cite{Elliott2019aa}]
  \thlabel{lem-luke-0}
  Let $S$ be a semigroup, let $\T$ be a semigroup topology for $S$, and let $A\subseteq S$. If $S$ has property \textbf{X}
  with respect to $A$, then the following hold:
  \begin{enumerate}[\rm (i)]
    \item
          \label{lem-luke-0-iii}
          if $\T$ is Polish and $A$ is a Polish subgroup of $S$, then $\T$ is
          maximal among the Polish semigroup topologies for $S$;

    \item
          \label{lem-luke-0-iv}
          if $A$ is a semigroup which has automatic continuity with respect to
          a class $\mathcal{C}$ of topological semigroups, then the semigroup
          $S$ has automatic continuity with respect to $\mathcal{C}$ also.
  \end{enumerate}
\end{theorem}

We start with a special case, that of $\omega \mathbb{K}_n$.  
\begin{prop}\label{prop-xxx}
  The monoid \(\End(\omega \mathbb{K}_n)\) has property~\textbf{X} with respect to the
  pointwise topology and \(\Aut(\omega \mathbb{K}_n)\) for every natural number \(n \geq 1\). The same result holds for \(\omega \mathbb{K}_n\) with loops.
\end{prop}
\begin{proof}

As will become apparent, the presence or absence of loops in $\omega\mathbb{K}_n$   is not relevant for the proof we are about to present. As such  $\omega\mathbb{K}_n$ should, in the mind of the reader, represent the structure with, or without, loops, according to their preference. 
It will be convenient to 
identify $\End(\omega \mathbb{K}_n)$ with the wreath product $\End(\mathbb{K}_n) \wr \N ^ \N$ where $\mathbb{K}_n$ either has loops or does not depending on the choice made in the previous sentence. It is clear that in this context $\Aut(\omega \mathbb{K}_n)$ is
$\Sym(\mathbb{K}_n)\wr \Sym(\N)$.

We will define the functions $f_s$ and $g_s$ from the definition of property \textbf{X} independently of $s$, and so we do not use the subscript in the remainder of this proof.

Let \(\phi:\N\to \N\) be an injective function with $\N\setminus (\N)\phi$ infinite. We define the element $f\in \End(\mathbb{K}_n) \wr \N ^ \N$ from the definition of property \textbf{X} to be:
\[f:= ((\id_{\mathbb{K}_n})_{i\in \N}, \phi).\]

If $\gamma \in \N^{\N}$ is any surjection such that the preimage of every point is infinite, then we define
  the second element $g\in \End(\mathbb{K}_n) \wr \N ^ \N$ from property \textbf{X} to be:
\[
g := ((h_i)_{i \in \N}, \gamma)
\]
for any sequence $(h_i)_{i \in \N}$ such that 
every set $A_{j, r} = \set{k\in \N}{(k)\gamma = j,\ h_k = r}$ is infinite where $j\in \N$ and $r\in \End(\mathbb{K}_n)$. 

For \(s := ((s_i)_{i \in \N}, \mu) \in \End(\mathbb{K}_n)\wr \N ^ \N\), we define \(t_s\in \Sym(\mathbb{K}_n) \wr \Sym(\N)\) to be 
\[
t_s = ((\id_{\mathbb{K}_n})_{i\in \N}, \nu)
\]
where $\nu$ is any permutation on $\N$ such that 
$(i)\phi\nu \in A_{(i)\mu,s_i}$ 
for every $i\in \N$.

It follows that 
\[
ft_sg = ((\id_{\mathbb{K}_n})_{i\in \N}, \phi)\ ((\id_{\mathbb{K}_n})_{i\in \N}, \nu)\ ((h_i)_{i\in \N}, \gamma)
=  ((\id_{\mathbb{K}_n})_{i\in \N}, \phi\nu)\ ((h_i)_{i\in \N}, \gamma)
= ((h_{(i)\phi\nu})_{i \in \N}, \phi\nu\gamma).
\]
If $i\in \N$ is arbitrary, then $(i)\phi\nu\in A_{(i)\mu, s_i}$ and so $((i)\phi\nu)\gamma = (i)\mu$ and $h_{(i)\phi\nu} = s_i$. Hence 
\begin{equation}\label{eq-ftsg}
ft_sg = ((h_{(i)\phi\nu})_{i \in \N}, \phi\nu\gamma) = ((s_i)_{i\in \N}, \mu)=s.
\end{equation}

Let $U$ be an arbitrary open neighbourhood of $t_s$. We may assume without loss of generality that 
\[
U = \set{((u_i)_{i \in\N}, \zeta)}{u_j = \id_{\mathbb{K}_n}\text{ and } (j)\zeta = (j)\nu \text{ for all } j\in F}\] for some finite subset $F$ of $\N$. We will show that $fUg$ is a neighbourhood of $s$. More precisely, we will show that the set $fUg$ and
\[
V := \set{((u_i)_{i \in\N}, \zeta)}{u_j = s_j \text{ and } (j)\zeta = (j)\mu \text{ for all } j\in (F)\phi ^ {-1}}
\]
coincide.
Note that if $v = ((u_i)_{i \in\N}, \zeta)\in V$, then $(j)\zeta = (j)\mu = (j)\phi\nu\gamma$ for all $j\in F$.
That $fUg\subseteq V$ follows after a routine computation. 

For the converse suppose that $v := ((u_i)_{i \in\N}, \zeta) \in V$. If
\(t_v\in \Sym(\mathbb{K}_n) \wr \Sym(\N)\) is such that
\[
t_v = ((\id_{\mathbb{K}_n})_{i\in \N}, \xi)
\]
where $\xi$ is any permutation on $\N$ such that $(i)\phi\xi \in A_{(i)\zeta, u_i}$ 
for every $i\in \N$, then, as in \eqref{eq-ftsg}, $v = ft_vg$.
To conclude the proof, it suffices to show that $\xi$ can be chosen such that $t_v\in U$.

If $j \in F\setminus\im(\phi)$, then the value of $(j)\xi$ does not affect the value of $ft_vg$.
In particular, we may choose $(j)\xi = (j)\nu$ for all $j\in F\setminus\im(\phi)$. On the other hand, if
 $j\in F \cap \im(\phi)$, then $(k)\phi = j$ for some $k\in \N$ and so
 $(j)\nu=(k)\phi \nu \in A_{(k)\mu, s_k}$.
 Since $k=(j)\phi^{-1}\in F \phi^{-1}$ and by the definition of $V$,
 $(k)\zeta = (k)\mu$ and $u_k = s_k$. Thus $A_{(k)\mu, s_k} = A_{(k)\zeta, u_k}$
 and we may let $(j)\xi=(j)\nu$ whenever $j\in \im(\phi) \cap F$.
\end{proof}

We will use 
Proposition~\ref{prop-xxx} to show that the 
endomorphism monoid $\End(\omega \mathbb{K}_n)$ of $\omega \mathbb{K}_{n}$ (with and without loops) has automatic continuity with respect to the class of second countable topological semigroups; see Theorem~\ref{thm-final}.

\subsection{A sufficient condition for property \textbf{X}}

  Suppose that \(\mathbb{A}\) and \(\mathbb{B}\) are structures in some fixed
  signature. Then a non-empty collection \(S\) of isomorphisms between finitely
  generated substructures of \(\mathbb{A}\) and \(\mathbb{B}\) is a \emph{back
  and forth system} if every \(f \in S\) satisfies:
  \begin{description}
    \item[Forth] if \(a \in A \setminus \dom(f)\) then there is \(g \in S\) such that
      \(f \subseteq g\) and \(a \in \dom(g)\);

    \item[Back] if \(b \in B \setminus \im(f)\) then there is \(g \in S\) such that
      \(f \subseteq g\) and \(b \in \im(g)\).
  \end{description}
  We will call a system a \emph{forth system} if it satisfies the first property and a \emph{back system} if it satisfies the second property above. 

\begin{lem}[cf. Section 3.2 in \cite{Hodges1997aa}]\label{lem:baf-iso}
  Suppose that \(\mathbb{A}\) and \(\mathbb{B}\) are at most countable
  structures in some fixed signature and let \(S\) be a back and forth system.
  If \(f \in S\), then there are \(f_0 := f, f_1, \ldots \in S\) such that
  \(f_0 \subseteq f_1 \subseteq \ldots\) and so that \(\bigcup_{i = 0}^\infty
  f_i\) is an isomorphism from \(\mathbb{A}\) to \(\mathbb{B}\).
\end{lem}

Throughout the remainder of this section we suppose that $\mathbb{X}$ is relational structure. 
If $f\in \Emb(\mathbb{X})$ and $g, h\in \End(\mathbb{X})$, then we define $S_{f, g, h}$ to be the set of finite partial
isomorphisms \(p\) of \(\mathbb{X}\) such that \(fp g \subseteq h\) and \(pg \cup f^{-1}h\) is a partial homomorphism. Note that $S_{f,g,h}$ is closed under restrictions of functions.

We will prove in a sequence of three lemmas, that subject to certain conditions on $f$ and $g$, $S_{f, g, h}$ is a back and forth system. There are three such conditions, which we refer to as \textbf{back}, \textbf{fo}, and \textbf{rth} (defined below) for lack of any more descriptive names.
These conditions are a little technical and they apply to a range of different relational structures. However, showing that these conditions hold is relatively straightforward in many instances. Proposition~\ref{prop-xxx} is a simple example which may be helpful to keep in mind when parsing the more general conditions.
    
\begin{de} \label{prop-structures-iv} 
  An endomorphism $g \in \End(\mathbb{X})$ is   
  \textbf{rth} if it is surjective and for all finite \(U \subseteq X\),
  \(u \in X \setminus U\), and \(v \in X\) such that  
  \(g{\restriction_{U}}\cup \{(u,v)\}\) is a homomorphism,
  there exists $\alpha$ in the stabiliser \( \Aut(\mathbb{X})_U\) with \((u)\alpha g=v\).
\end{de}
      
\begin{lem}\label{lem-flex}
  Let $f\in \Emb(\mathbb{X})$ and let  $g, h\in \End(\mathbb{X})$. 
 If $g$ is \textbf{rth} and $p\in S_{f, g, h}$ and $a\in \im(f)\setminus \dom(p)$, then there exists  $q\in S_{f,g,h}$ such that $p\subseteq q$ and $a\in \dom(q)$.
\end{lem}

\begin{proof}
 Since \(\mathbb{X}\) is
    homogeneous, there exists
    \(b \in X \setminus \im(p)\) such that \(p \cup \{(a, b)\} \) is an
    isomorphism. Since \(p \in S_{f, g, h}\), we have that \(pg \cup \{(a, (a)f^{-1}h)\} \subseteq pg \cup f^{-1}h\) is a homomorphism. Multiplying by the inverse of \(p \cup \{(a, b)\} \) on the left, we obtain the homomorphism \(g{\restriction_{\im(p)}}\cup \{(b, (a)f^{-1}h)\}\).
    Since $g$ is \textbf{rth}, 
    there is \(\alpha \in \Aut(\mathbb{X})_{\im(p)}\) so that  \((b)\alpha g = (a)f^{-1}h\).
    We define \(q := (p \cup \{(a, b)\}) \circ \alpha\). By definition, $q$ is also an isomorphism.
    If $x\in \dom(fqg)$, then 
    \[(x)fqg
    = \begin{cases}
      (x)fp\alpha g = (x)fpg = (x)h & \text{if } (x)f\in\dom(p) \\
      (b)\alpha g = (a)f ^ {-1}h = (x)h & \text{if } (x)f = a
    \end{cases}
    \]
    and so \(fqg  \subseteq h\).  Finally, 
    \(qg \cup f^{-1}h = pg \cup f^{-1}h\), and
    therefore, \(q \in S_{f, g, h}\), \(p \subseteq q\), and \(a \in \dom(q)\), as required.
\end{proof}

\begin{de}\label{prop-structures-iii}
We say that $f\in \End(\mathbb{X})$ is \textbf{fo} if $f$ is an embedding and 
for all finite $Y\subseteq X \setminus \im(f)$ there exists finite $W_Y\subseteq \im(f)$ such that the following holds. If
$t:\im(f)\cup Y\to X$ is a function such that $t{\restriction_{\im(f)}}$ and $t{\restriction_{(Y\cup
W_Y)}}$ are homomorphisms, then $t$ is a homomorphism.
\end{de}

\begin{lem}\label{lem-turbo-flex}
  Let $f, g\in \End(\mathbb{X})$. 
  If $f$ is \textbf{fo} and $g$ is \textbf{rth}, 
  then $S_{f, g, h}$ is a forth system for all $h\in \End(\mathbb{X})$.
\end{lem}

\begin{proof}
Let $p\in S_{f, g, h}$ be arbitrary and let $a\in X \setminus \dom(p)$.
    By Lemma \ref{lem-flex}, it suffices to consider the case when $a\in X \setminus \left(\im(f)\cup \dom(p)\right)$.
    If $b\in W_{\{a\} \cup \dom(p) \setminus \im(f)}$  and $b\not\in \dom(p)$, then 
    by Lemma \ref{lem-flex}, there exists $p'\in S_{f, g, h}$ such that $p\subseteq p'$ and $b\in \dom(p')$. Note that $W_{\{a\} \cup \dom(p) \setminus \im(f)} = W_{\{a\} \cup \dom(p') \setminus \im(f)}$. In this way, by repeatedly extending $p$, 
    we may assume without loss of generality that 
    \begin{equation}\label{WAdomp}
     W_{\{a\} \cup  \dom(p) \setminus \im(f)} \subseteq \dom(p) \cap \im(f)
    \end{equation}

Since \(\mathbb{X}\) is homogeneous,
    there is \(b \in X\) such that the map \(q := p \cup \{(a, b)\} \) is an
    isomorphism. We will show that 
 \(q \in S_{f, g, h}\) and so $S_{f, g, h}$ is a forth system. 
    It follows immediately from the definition that \( fqg=fpg\subseteq h\).
    
    It remains to prove that $t:=qg \cup f ^ {-1}h$ is a homomorphism. 
    Since $f$ is \textbf{fo}, it suffices to show that $Y := \dom(t)\setminus\im(f) \subseteq X \setminus \im(f)$ has the property that $t{\restriction_{\im(f)}}$ and $t{\restriction_{Y\cup W_Y}}$ are homomorphisms. The set $Y$ is finite because $t$ is a finite extension of $f ^ {-1}h$. 
    Since \( t{\restriction_{\im(f)}} = f^{-1}h\), it follows that $t{\restriction_{\im(f)}}$  is a homomorphism.
    
    Clearly, $Y = \dom(t)\setminus \im(f) = \dom(q)\setminus\im(f) \subseteq \dom(q) = \dom(qg)$.
    Similarly, $\dom(t) = \dom(q) \cup \im(f) = \{a\} \cup \dom(p) \cup \im(f)$ and so 
    $Y = \{a\} \cup \dom(p)\setminus \im(f)$. In particular, $W_Y =  W_{\{a\} \cup  \dom(p) \setminus \im(f)} \subseteq \dom(p) \subseteq \dom(q)= \dom(qg)$ by \eqref{WAdomp}.
    Thus \(t{\restriction_{W_Y \cup Y}} \subseteq qg\)
    is a homomorphism, as required.
\end{proof}

\begin{de}\label{prop-structures-ii}
We say that $f\in \Emb(\mathbb{X})$ is \textbf{back} if for every finite
       substructure \(\mathbb{T}\) of \(\mathbb{X}\) there exists a finite substructure \(\mathbb{U}\) of
      \(\mathbb{X}\) such that \(T \subseteq U\) and for all \(v\in X \backslash U\) the following hold:
      \begin{enumerate}[\rm (i)]
          \item 
           there exists \(\alpha \in \Aut(\mathbb{X})_U\) such that
      \((v)\alpha \notin \im(f)\)
      \item 
      if \(t \colon X \to X\) is a finite partial map such that 
      \((U\cup \{v\})\alpha
      \subseteq \dom(t) \subseteq (U\cup \{v\})\alpha \cup \im(f)\),
      and \(t\) restricted to the sets \(\dom(t)\setminus\{(v)\alpha\}\) and
      \((U\cup \{v\})\alpha\) are homomorphisms,
      then \(t\) is a homomorphism.
      \end{enumerate}
\end{de}

\begin{lem}\label{lem-flux}
If $f\in \End(\mathbb{X})$ is \textbf{back} and \textbf{fo}, and $g\in \End(\mathbb{X})$ is \textbf{rth}, then $S_{f, g, h}$ is a back and forth system for all $h\in \End(\mathbb{X})$.  
\end{lem}
\begin{proof}
    By Lemma \ref{lem-turbo-flex}, we  need only  show that $S_{f,g,h}$ is a back system. Let $p\in S_{f,g,h}$ and \(b \in X \setminus \im(p)\) be arbitrary. Let
    \(\mathbb{T}\) be the induced substructure of \(\mathbb{X}\) such that \(T
    = \dom(p)\), and let \(\mathbb{U}\) be as in~Definition~\ref{prop-structures-ii}.
    Since $S_{f,g,h}$ is a forth system, as shown in Lemma \ref{lem-turbo-flex}, we may extend \(p\) to \(p' \in S_{f, g, h}\) such that \(U \subseteq
    \dom(p')\). Moreover, since \(S_{f, g, h}\) is closed under restrictions, there is \(p'' \in S_{f, g, h}\) extending \(p\) with \(U = \dom(p'')\). 
    Hence we may assume without loss of generality that  \(U = \dom(p)\). 
        Since \(\mathbb{X}\) is homogeneous,
    there exists \(v \in X\) such that \(p \cup \{ (v,b)\}\) is an
    isomorphism.  Since $f$ is \textbf{back}, there exists \(\alpha \in \Aut(\mathbb{X})\) as in Definition~\ref{prop-structures-ii}(i).
    Define a partial isomorphism \(q := \alpha^{-1} \circ (p \cup \{ (v,
    b)\})\). Since \( (v)\alpha \notin \im(f)\), it follows that \(fpg=fqg \subseteq h\).

    To conclude the proof we will show that
     \(t = qg \cup
    \{ (z, (z)f^{-1}h) \colon z\in Z\}\) is a homomorphism for every finite subset $Z$ of $\im(f)$.
    Since $fqg \subseteq h$, it follows that \(t\) is a
    well-defined partial map. Since $f$ is \textbf{back}, it suffices, by Definition~\ref{prop-structures-ii}(ii), to show that 
    $t|_{\dom(t)\setminus (v)\alpha}$ and $t|_{(U \cup \{v\})\alpha}$ are homomorphisms.
    By definition \( (U\cup \{v\})\alpha = \dom(q)\) and so \( (U\cup \{v\})\alpha \subseteq
    \dom(t) \subseteq (U \cup \{v\})\alpha \cup \im(f)\).
    Hence
    \(t{\restriction_{(U \cup \{v\})\alpha}} = qg\) and the latter is clearly a homomorphism. Finally, 
    \[
      t{\restriction_{\dom(t) \setminus \{(v)\alpha\}}} = pg \cup
    \{ (z, (z)f^{-1}h) \colon z\in Z\}
    \]
    is a homomorphism since \(p \in S_{f,g,h}\). Hence \(t\) is
    a homomorphism, as required. 
\end{proof}

Finally, we state the main result of the section.

\begin{theorem}\label{prop-structures-property-x}
  Let $\mathbb{X}$ be a homogeneous relational structure and let 
  $f, g\in \End(\mathbb{X})$.
  If
  $f$ is \textbf{back} and \textbf{fo}, and $g$ is \textbf{rth},
  then \(\End(\mathbb{X})\) equipped with the pointwise topology has property
  \textbf{X} with respect to \(\Aut(\mathbb{X})\).
\end{theorem}

\begin{proof}
  Let \(s \in \End(\mathbb{X})\) be fixed. By Lemma \ref{lem-flux} and the assumptions of the theorem, $S_{f, g, s}$ is a back and forth system.
  It follows from Lemma~\ref{lem:baf-iso} applied to \(\varnothing \in S_{f, g, s}\) that there is \(\alpha_s \in \Aut(\mathbb{X})\) such that
  \(s = f \alpha_s g\). Let $\mathcal{B}$ be the basis for the pointwise topology consisting of finite intersections of subbasic open sets $\set{f\in \End(\mathbb{X})}{(x)f=y}$ for any $x,y\in \mathbb{X}$. If \(B \in \mathcal{B}\) contains \(\alpha_s\), then there is a finite \(F \subseteq X\) such that \(B =
  \{h \in \End(\mathbb{X}) \colon h{\restriction_F} = \alpha_s{\restriction_F}\}\).
  Since $f$ is \textbf{fo}, there exists $W_{F \setminus \im(f)} \subseteq \im(f)$ as given in Definition \ref{prop-structures-iii}. Let
  \(
    H = F \cup W_{F \setminus \im(f)}.
  \)
  It suffices to prove that \( f (B \cap \Aut(\mathbb{X})) g\) is a
  neighbourhood of \(s\). In particular, we will show that 
  \begin{equation}\label{mystery-claim}
    \{h \in \End(\mathbb{X}) \colon h{\restriction_{(H)f^{-1}}} =
    s{\restriction_{(H)f^{-1}}}\} \subseteq f (B \cap \Aut(\mathbb{X})) g.
  \end{equation}
  Let \(h \in \End(\mathbb{X})\) be such that \(h{\restriction_{(H)f^{-1}}} =
  s{\restriction_{(H)f^{-1}}}\) and let \(p = \alpha_s{\restriction_H}\). 
  If $x\in (H)f ^ {-1}$, then $(x)f \in H$ and so $(x)f\alpha_s = (x)fp$. 
  Since \(s = f\alpha_sg\), it follows that 
  \(h{\restriction_{(H)f^{-1}}} = s{\restriction_{(H)f^{-1}}} = fpg\). Since $S_{f, g, h}$ is a back and forth system, it follows from Lemma~\ref{lem:baf-iso} that if
  \(p \in S_{f, g, h}\), then \(p\) can be extended to \(\alpha_h \in \Aut(\mathbb{X}) \cap
  B\) such that \(h = f \alpha_h g \in f (\Aut(\mathbb{X}) \cap B) g\), proving~\eqref{mystery-claim}. It remains to show that \(p \in S_{f, g, h}\).  We have already shown that \(fpg \subseteq h\), and so the proof is concluded by showing that $t := pg \cup f^{-1}h$ is a homomorphism.

  First, \(t\) is well-defined, since \(f\alpha_sg = s\), and so the three functions \(pg\), $\alpha_sg = f^{-1}s$, and $f^{-1}h$, all coincide on
  \(\dom(pg) \cap \im(f) = H \cap \im(f)\). 
  Since
  $
     \dom(t) \setminus \im(f) = H \setminus \im(f) = F \setminus \im(f),
  $
  it follows that \( W_{\dom(t) \setminus \im(f)} = W_{F \setminus \im(f)} \subseteq H \subseteq \dom(t)\).
  Finally, \(t{\restriction_{ \im(f)}} = f^{-1}h\) and
  $
    t{\restriction_{W_{\dom(t) \setminus \im(f)} \cup (\dom(t) \setminus
    \im(f))}} \subseteq t{\restriction_H} = pg
  $
  are both homomorphisms. Therefore, since \(f\) is \textbf{fo} (Definition~\ref{prop-structures-iii}), \(t\) is a homomorphism
  and so \(p \in S_{f, g, h}\), as required. 
\end{proof}

\subsection{Strong amalgamation property with homomorphism gluing}
Next, we introduce a strengthening of the classical strong amalgamation property called \emph{the strong amalgamation property with homomorphism gluing}.
This property can be thought of as an ``almost free amalgamation property''.  

\begin{de} \label{definition-homo-glue}
  Let \(\sigma\) be a finite relational signature. Then  a class of  finite
  \(\sigma\)-structures \(\mathcal{C}\) has the \emph{strong amalgamation property
  with homomorphism gluing} if for any \(\mathbb{A}, \mathbb{B}, \mathbb{C}
  \in \mathcal{C}\) and  any embeddings \(e_1 \colon \mathbb{A} \to
  \mathbb{B}\) and \(e_2 \colon \mathbb{A} \to \mathbb{C}\), there exist
  \(\mathbb{D} \in \mathcal{C}\) and embeddings \(f_1 \colon \mathbb{B} \to
  \mathbb{D}\) and \(f_2 \colon \mathbb{C} \to \mathbb{D}\) such that:
  \begin{enumerate}[\rm (i)]
    \item \(e_1 f_1 = e_2 f_2\) and \(\im(f_1)\cap \im(f_2) =
      (A)e_1f_1 = (A)e_2f_2\); 

    \item if \(g \colon D \to E\) is a partial map for some
      \(\mathbb{E} \in \mathcal{C}\) so that \((A)e_1f_1 =(A)e_2f_2 \subseteq
      \dom(g)\) and the maps \(f_1 g\) and \(f_2 g\) are homomorphisms, then
      \(g\) is also a homomorphism.
  \end{enumerate}
\end{de}

Interestingly, and unbeknownst to us at the time of our research, the
similar notion of \emph{canonical amalgamation} has recently been
introduced in~\cite{paolini2018strong} in the context of the
reconstruction of the action of automorphism groups from their
abstract group structure. Roughly speaking, the canonical amalgamation property
stipulates the existence of a strong amalgam in which any two tuples
which would be belong to the same orbit in the free amalgam (which
might not be a member of the class under consideration) also belong to
the same orbit in the canonical amalgam. The strong amalgamation
property with homomorphism gluing implies the canonical amalgamation
property, and stipulates the existence of a strong (and in fact
canonical) amalgam in which all positive facts are implied by the
positive facts of the free amalgam plus the membership in the class
under consideration. We remark that the obvious ``first-order'' analogue
of homomorphism gluing where one would replace homomorphisms by
embeddings would not yield a fruitful notion as it would imply a
unique strong amalgam.

Examples of structures whose age has the canonical amalgamation property
but not the strong amalgamation property with homomorphism gluing are
the universal homogeneous tournament, the dual of the universal
homogeneous triangle-free graph, and the usual linear order on the rational
numbers. Clearly, the free amalgamation property implies the strong
amalgamation property with homomorphism gluing. The age of the random strict partial
order does not have free amalgamation, but does have the strong
amalgamation property with homomorphism gluing, as shown in the next example.

\begin{ex}\label{example-P-has-saphg}
  We will show that the age of the random strict partial order has the strong amalgamation property with homomorphism gluing. Suppose that \(\mathbb{A}, \mathbb{B}, \mathbb{C}\) are finite
  partial orders and \(e_1 \colon \mathbb{A} \to \mathbb{B}\) and \(e_2 \colon \mathbb{A} \to \mathbb{C}\) are embeddings. It suffices to consider the case when \(e_1\) and \(e_2\) are both identity maps and \(B \cap C = A\). Let \(<\) be a binary relation on \(B\cup C\) given by \(x < y\) if either \(x, y \in B\) and \(x <^{\mathbb{B}} y\); or \(x, y \in C\) and \(x <^{\mathbb{C}} y\). Note that \(<\) has no cycles, and so the transitive closure of \(<\) is a partial order. Let
  \(\mathbb{D}\) be a finite partial order with domain \(B \cup C\) and let
  \(<^{\mathbb{D}}\) be the transitive closure of \(<\). If \(f_1 \colon B \to D\) and \(f_2 \colon C \to D\) are the identity maps, then (i) of Definition~\ref{definition-homo-glue} holds.
  
  Suppose that \(g \colon D \to E\) is a map as in the hypothesis of Definition~\ref{definition-homo-glue}(ii). Let \(x, y \in D\) be such that \(x <^{\mathbb{D}} y\). If both \(x\) and \(y\) are in either \(B\) or \(C\), then \((x)g <^{\mathbb{E}} (y)g\) since \(g\) is a homomorphism when restricted to either \(B\) or \(C\). 
  Suppose that \(x \in B\) and \(y \in C\setminus B\), the other cases are similar, and so will be omitted. Since \(<^{\mathbb{D}}\) is the transitive closure of \(<\), it follows that there is \(z \in A\) such that \(x < z < y\).
  Therefore,  \(x <^{\mathbb{D}} z <^{\mathbb{D}} y\), and so  \((x)g <^{\mathbb{E}} (z)g <^{\mathbb{E}} (y)g\), proving that \(g\) is a homomorphism.
  
  A similar argument shows that the random reflexive partial order also has the strong amalgamation property with homomorphism gluing.
\end{ex}

Next, we show that one of the structures in the definition of the strong amalgamation property with homomorphism gluing can be allowed to be countably infinite. 

\begin{lem}\label{lem-homo-glue}
Let $\sigma$ be a relational signature and 
  let \(\mathcal{C}\) be a
  Fra\"{i}ss\'e class of \(\sigma\)-structures which has the strong amalgamation
  property with homomorphism gluing whose Fra\"{i}ss\'{e} limit \(\mathbb{X}\) is $\omega$-categorical.
  If \(\mathbb{A}, \mathbb{B} \in \mathcal{C}\), \(\mathbb{C}\) is a countable
  \(\sigma\)-structure which embeds into \(\mathbb{X}\), and if \(e_1 \colon
  \mathbb{A} \to \mathbb{B}\) and \(e_2 \colon \mathbb{A} \to \mathbb{C}\) are
  embeddings, then there exist \(\mathbb{D}\), which embeds into \(\mathbb{X}\), and embeddings
  \(f_1 \colon \mathbb{B} \to \mathbb{D}\) and \(f_2 \colon \mathbb{C} \to
  \mathbb{D}\) such that:
  \begin{enumerate}[\rm (i)]
    \item\label{lem-homo-glue-ii} \(e_1 f_1 = e_2 f_2\), \(\im(f_1)\cup \im(f_2)=D\) and \(\im(f_1)\cap \im(f_2) =
      (A)e_1f_1 = (A)e_2f_2\);

    \item\label{lem-homo-glue-iii} if \(g \colon D \to E\) is
      a partial map for some \(\sigma\)-structure \(\mathbb{E}\) which embeds
      into \(\mathbb{X}\) so that \((A)e_1f_1 =(A)e_2f_2 \subseteq \dom(g)\) and the
      maps \(f_1 g\) and \(f_2 g\) are
      homomorphisms, then \(g\) is also a homomorphism.
  \end{enumerate}
\end{lem}

\begin{proof}
  Let \(\mathbb{A}\), \(\mathbb{B}\), and \(\mathbb{C}\) be as in the
  hypothesis of the lemma. We may assume, for the ease of notation, that
  \(\mathbb{A}\) is a substructure of both \(\mathbb{B}\) and \(\mathbb{C}\), that $e_1, e_2$ are the identity function on $\mathbb A$,
  and also that both \(\mathbb{B}\) and \(\mathbb{C}\) are substructures of
  \(\mathbb{X}\). Let \(C_0 := A \subseteq C_1 \subseteq \cdots\) be finite
  sets such that \(\bigcup_{k = 0}^\infty C_k = C\), and for every \(k \geq 0\) let
  \(\mathbb{C}_k\) be the induced substructure of \(\mathbb{C}\) on \(C_k\).

Since \(\mathcal{C}\)
  has the strong amalgamation property with homomorphism gluing and by the homogeneity of $\mathbb X$, 
  if we apply Definition~\ref{definition-homo-glue} to $A$, $B$, and $C_k$ with the identity function from $A$ into both $B$  and $C_k$, 
  it follows that for every $k$
  there is a substructure \(\mathbb{D}_k\) (the amalgam of $B$ and $C_k$ from Definition~\ref{definition-homo-glue}) of \(\mathbb{X}\) and an embedding  $h_k\in\Emb(\mathbb X)$ 
  which fixes $A$
pointwise and such that $B\cup (C_k)h_k = D_k$. 

We use a similar argument to that used several times in Section~\ref{section-fraisse} (for example in the proofs of Lemmas~\ref{lem-zariski-functions1} and~
\ref{lem-zariski-functions2}). 
If $\alpha_k\in \Aut(\mathbb{X})_{B}$ is arbitrary, then $h_k\alpha_k$ is an embedding that fixes $A$ pointwise, and such that $B\cup (C_k)h_k\alpha_k$ induces a substructure of $\mathbb{X}$ isomorphic to $\mathbb{D}_k$. 
Since $\mathbb{X}$ is $\omega$-categorical, there are 
automorphisms $\alpha_k\in \Aut(\mathbb{X})_B$ such that a subsequence of $(h_k\alpha_k)_{k\in \N}$ converges to an embedding $f_2\in \Emb(\mathbb X)$.  
The proof is concluded by 
setting $\mathbb{D}$ to be the structure induced by $B\cup (C)f_2$, and setting $f_1$ to be $\id_{B}: \mathbb{B}\to \mathbb{D}$; the claimed properties~(i) and~(ii) are straightforward to verify.

\end{proof} 

\subsection{Application of the sufficient conditions}
In this section, we use the sufficient conditions developed in the previous sections to show that the endomorphism monoids of structures in a certain class have property \textbf{X} with respect to their automorphism groups. The main theorem in this section is the following.

\theomodeltheorypropertyx*

Note that Proposition~\ref{prop-xxx} is not a special case of Theorem~\ref{theomodeltheorypropertyx} because $\omega \mathbb{K}_n$ does not have the strong amalgamation property when $n\geq 2$.
Before proceeding with the proof of the above theorem, we prove the following
corollary. 

We remark that the age of $n\mathbb{K}_{\omega}$ without loops does not have strong amalgamation with homomorphism gluing, which is why this structure does not appear in the statement of the next corollary.

\begin{cor}\label{cor-property-x}
  The endomorphism monoids of the following structures have property \textbf{X}
  with respect to the pointwise topology and the corresponding automorphism group:
  \begin{enumerate}[\rm (i)]
    \item the random graph;
    \item the random directed graph;
    \item the random strict partial order;
    \item the graph \(n \mathbb{K}_\omega\) with loops for $n\in \N\cup \{\omega\}$, i.e., the random equivalence relations with \(n\) countably infinite equivalence classes;
  \end{enumerate}
  as well as (i)-(iii) with all the loops included. 
\end{cor}

\begin{proof}
  All of the structures mentioned in the hypothesis are classical examples of
  homogeneous structures. Moreover, the random graph, the random directed
  graph, and the random partial order are well-known to be
  homomorphism-homogeneous, see~\cite{Cameron2006aa, Coleman:2017aa}. It is
  routine to show that \(n \mathbb{K}_{\omega}\) is homomorphism-homogeneous and that
  the aforementioned structures with all the loops added are still
  homomorphism-homogeneous. 
  
  The random graph and the random digraph have the free amalgamation property,
  and so also have the strong amalgamation property with homomorphism gluing. We showed in Example~\ref{example-P-has-saphg} the age of the random strict partial order has 
  the strong amalgamation property with homomorphism gluing. The same arguments work for the corresponding structures with
  loops. It is straightforward to show 
  directly that the age of \(n \mathbb{K}_{\omega}\) with loops for \(n
  \leq \omega\) has the strong amalgamation property with homomorphism gluing. 
\end{proof}

In the proof of Theorem~\ref{theomodeltheorypropertyx}, in particular
Claim~\ref{claim-g-exists}, we will use a somewhat non-classical version of
Fra\"{i}ss\'{e} theory. In particular, besides constants, relations, and
functions we will allow partial functions. Under this generalisation,
however, Fra\"{i}ss\'{e}'s Theorem remains unchanged. This can be achieved
by adding a constant \(\na\) to the signature and working with total functions
which yield \(\na\) whenever the corresponding partial function is undefined. 



\begin{proof}[Proof of Theorem~\ref{theomodeltheorypropertyx}]
  We proceed by showing that the structure \(\mathbb{X}\) satisfies the
  hypothesis of Theorem~\ref{prop-structures-property-x}. We begin by
  showing that the required \(f \in \Emb(\mathbb{X})\) and \(W_F \subseteq
  \im(f)\) for every finite \(F \subseteq X \setminus \im(f)\) exist.

  \begin{claim}\label{claim-inside-proof}
    There exists a \textbf{back} and \textbf{fo} embedding \(f \in \Emb(\mathbb{X})\).
\end{claim}

  \begin{proof}
    We will construct a countable \(\sigma\)-structure \(\mathbb{Y}\) which has
    a self-embedding that is \textbf{back} and \textbf{fo}, and show that the
    structure we have constructed is indeed isomorphic to \(\mathbb{X}\). 
    We define $\mathbb{Y}$ as a union $\bigcup_{n \in \N} \mathbb{Y}_n$ of structures $\mathbb{Y}_n$ which we define by induction, starting with $\mathbb{Y}_0 := \mathbb{X}$. We fix any bijection $F : \N\times \N \to \N\setminus \{0\}$  such that $(i, j)F > i$ for all $i, j\in \N$.  At every step of the induction, right after the construction of $\mathbb Y_n$ we moreover fix a sequence $(\mathbb{U}_{n,k},\mathbb{V}_{n,k})_{k\in \N}$ of pairs such that for all $k\in \N$ we have that  $\mathbb{U}_{n,k}$ is a substructure of both $\mathbb Y_n$ and $\mathbb{V}_{n,k}\in\Age(\mathbb X)$, and such that the following holds: for every finite substructure $\mathbb{A}$ of $\mathbb{Y}_n$ and every finite extension of $\mathbb{B}$ of $\mathbb{A}$ in $\Age(\mathbb X)$, there exists $k\in \N$ such that $\mathbb{U}_{n, k} = \mathbb{A}$ and such that  $\mathbb{V}_{n,k}$ is isomorphic to $\mathbb{B}$ via an isomorphism fixing \(\mathbb{U}_{n, k}\) pointwise. 
    \ignore{
   We define $\mathbb{Y}$ inductively starting with $\mathbb{Y}_0 := \mathbb{X}$. Additionally we require a sequence $(\mathbb{U}_{0,i},\mathbb{V}_{0,i})_{i\in \N}$ of pairs of substructures of $\mathbb{Y}_0$ and $\mathbb{X}$ such that
   for every finite substructure $\mathbb{A}$ of $\mathbb{Y}_0$ and every finite extension of $\mathbb{B}$ of $\mathbb{A}$ in $\mathbb{X}$, there exists $j\in \N$ such that $\mathbb{U}_{0, j} = \mathbb{A}$ and $\mathbb{V}_{0,j}= \mathbb{B}$. 
   }
   
    If $n\in \N$ with $n > 0$ and if $Y_i$ has been
constructed for all $i\in \N$ such that $i<n$, then we construct $\mathbb{Y}_n$ as follows.
Let $i, j\in \N$ be such that  $(i, j)F = n$. Then $i < n$. 
If $\mathbb{A} = \mathbb{U}_{i,j}$, $\mathbb{B} = \mathbb{V}_{i, j}$, and $\mathbb{C} = \mathbb{Y}_{n - 1}$, then we denote the amalgam $\mathbb{D}$ from Lemma~\ref{lem-homo-glue} by $\mathbb{Y}_n$ (here \(e_1\) and \(e_2\) are the identity maps). 
We may suppose that the function $f_2 : \mathbb{Y}_{n -1} \to \mathbb{Y}_{n}$ in Lemma~\ref{lem-homo-glue} is the identity mapping.  
\ignore{
This step in the induction is concluded by 
taking $(\mathbb{U}_{n,k},\mathbb{V}_{n,k})_{k\in \N}$ to be pairs of substructures of $\mathbb{Y}_{n}$ and $\mathbb{X}$ such that
   for every finite substructure $\mathbb{A}$ of $\mathbb{Y}_n$ and every finite extension of $\mathbb{B}$ of $\mathbb{A}$ that embeds in $\mathbb{X}$, there exists $k\in \N$ such that $\mathbb{U}_{n, k} = \mathbb{A}$, \(\mathbb{U}_{n, k}\leq \mathbb{V}_{n, k}\) and $\mathbb{V}_{n,k}$ is isomorphic to $\mathbb{B}$ via an isomorphism fixing \(\mathbb{U}_{n, k}\) pointwise.
   }
   
   Since $\mathbb{X}$ is a substructure of $\mathbb{Y}$, the age of $\mathbb{X}$ is contained in $\mathbb{Y}$.
   On the other hand, by Lemma~\ref{lem-homo-glue}, every $\mathbb{Y}_n$ embeds in $\mathbb{X}$
   and so every finite substructure of $\mathbb{Y}$ is also a substructure of $\mathbb{X}$.

The set of finite partial isomorphisms of $\mathbb{Y}$
is a back and forth system since both the domain and the range of any such partial isomorphism $f$ are
extended in all possible ways in the construction of $\mathbb{Y}$. Hence, $\mathbb{Y}$ is homogeneous
and so isomorphic to $\mathbb{X}$. Let $f\in \Emb(\mathbb{Y})$ be such that $\im (f) = \mathbb{X}=\mathbb Y_0$. We will show that $f$ is \textbf{back} and \textbf{fo}. 
  
\textbf{back:} If $\mathbb{T}$ is any finite substructure of $\mathbb{Y}$, then we set $\mathbb{U}$ in Definition~\ref{prop-structures-ii} to be $\mathbb{T}$. If $v\in Y \setminus U$ and $\mathbb{W}$ is the substructure of $\mathbb{Y}$ induced by $U\cup \{v\}$, then 
there exist \(i, j\in \N\) such that \(\mathbb{U}_{i,j}=\mathbb{U}\) and \(\mathbb{V}_{i,j}\) is isomorphic to $\mathbb{W}$ via an isomorphism $\beta: \mathbb{W} \rightarrow \mathbb{V}_{i,j}$ fixing \(U\) pointwise. If $n = (i, j)F$, 
 then the required map \(\alpha\) in Definition~\ref{prop-structures-ii} is any automorphism of $\mathbb{Y}$ extending $\beta f_1$ where \(f_1: \mathbb{V}_{i,j} \to \mathbb{Y}_n\) is the map from Lemma~\ref{lem-homo-glue} used in the construction of \(\mathbb{Y}_{n}\). 
 It is possible to verify that $\alpha$ satisfies the criteria in Definition~\ref{prop-structures-ii} and hence is \textbf{back}.
 
\textbf{fo:} It suffices to prove that if $n\in \N$ and $T\subseteq Y_n\setminus X= Y_n\setminus \im(f)$ is finite and $T \not \subseteq Y_{n-1}$, then \(W_T:=\bigcup_{i\leq n} U_{(i)F ^ {-1}}\cap X\) satisfies Definition~\ref{prop-structures-iii}. In other words, if $t:X \cup T \rightarrow Y$ is a function such that $t\restriction_{X}$ and $t\restriction_{T \cup W_T}$ are homomorphisms, then $t$ is a homomorphism.

We proceed by induction on $n$. The base case $n=0$ is trivial since $T=\varnothing$. So suppose that $n>0$, that the inductive hypothesis holds for all $i<n$, and that $T$ and $t$ are as above. By homomorphism-homogeneity there is $t'\in\End(\mathbb Y)$  extending \(t\restriction_{T\cup W_T}\). 
 If we define \(t'':= t \cup t'\restriction_{Y_n\backslash X} \), then $t'':Y_n \to Y$ 
 is a function by definition of $t'$. It suffices to show that \(t''\) is a homomorphism. Note that \(t''\restriction_{X}=t\restriction_{X}\) and \(t''\restriction_{(Y_n\backslash X) \cup W_T}=t'\restriction_{(Y_n\backslash X) \cup W_T}\) are both homomorphisms.
 
 Then, by the inductive hypothesis applied to $Y_{n - 1} \setminus X$ and $W_{Y_{n - 1} \setminus X}\subseteq W_T$, since  
 \[
( Y_{n - 1}\setminus X) \cup W_{Y_{n - 1} \setminus X}\subseteq (Y_n\backslash X) \cup W_T\] 
 $t''\restriction_{X}$ and $t''\restriction_{(Y_{n - 1}\setminus X) \cup W_{Y_{n - 1} \setminus X}}$
 are homomorphisms,
it follows that the
restriction of \(t''\) to \(Y_{n-1}\) is a homomorphism.
By the definition of
homomorphism gluing, the
restriction of \(t''\) to \(Y_n\) (which is \(t''\)) is a homomorphism.
  \end{proof}

  It remains to show that the there is \(g \in \End(\mathbb{X})\) as in
  Theorem~\ref{prop-structures-property-x}.

  \begin{claim}\label{claim-g-exists}
    There exists a \textbf{rth} \(g \in \End(\mathbb{X})\). 
  \end{claim}

  \begin{proof}
    Let \(\tau\) be the expansion of the signature \(\sigma\) by a partial function symbol
    \(\alpha\) and two unary relation symbols \(L\) and \(R\). Let \(\mathcal{K}\) be
    the class consisting of all finitely generated \(\tau\)-structures
    \(\mathbb{A}\) such that: 
    \begin{enumerate}[\rm (i)]
      \item \(L^{\mathbb{A}}\) and \(R^{\mathbb{A}}\) partition \(A\);
  
      \item \(\reduct{L}{A}\) and \(\reduct{R}{A}\), the
        \(\sigma\)-reducts of the induced substructures of \(\mathbb{A}\) on
        \(L^{\mathbb{A}}\) and \(R^{\mathbb{A}}\) respectively, are in
        \(\Age(\mathbb{X})\); 
  
      \item 
      $\alpha^{\mathbb{A}}: \reduct{L}{A} \to \reduct{R}{A}$ is a homomorphism and $\dom(\alpha^{\mathbb{A}}) = L^{\mathbb{A}}$.
    \end{enumerate}
    We emphasize that although \(\alpha^{\mathbb{A}}\) is a partial function on
    \(\mathbb{A}\), it is a full function on \(\reduct{L}{A}\). Next,
    we show that \(\mathcal{K}\) is a Fra\"{i}ss\'{e} class. 

    It is routine to verify that \(\mathcal{K}\) has the hereditary property, since
    \(\Age(\mathbb{X})\) has the hereditary property. Next, we will use strong
    amalgamation with homomorphism gluing to show that
    \(\mathcal{K}\) has the amalgamation property. Let \(\mathbb{A},
    \mathbb{B}, \mathbb{C} \in \mathcal{K}\), and let \(e_1 \colon \mathbb{A}
    \to \mathbb{B}\) and \(e_2 \colon \mathbb{A} \to \mathbb{C}\) be
    embeddings. For the ease of notation, we may assume that \(\mathbb{A}\) is
    a substructure of \(\mathbb{B}\) and \(\mathbb{C}\) and that \(e_1, e_2\)
    are identity maps. Since \(\reduct{R}{A}, \reduct{R}{B}, \reduct{R}{C} \in
    \Age(\mathbb{X})\) and \(\reduct{R}{A}\) is a substructure of both
    \(\reduct{R}{B}\) and \(\reduct{R}{C}\), there exists \(\mathbb{D}_R \in
    \Age(\mathbb{X})\), \(f_{R, 1} \colon \reduct{R}{B} \to \mathbb{D}_R\), and
    \(f_{R,2} \colon \reduct{R}{C} \to \mathbb{D}_R\) witnessing the strong
    amalgamation property with homomorphism gluing of \(\Age(\mathbb{X})\).
    Analogous witnesses \(\mathbb{D}_L\), \(f_{L,1}\), and \(f_{L,2}\) exist
    for the triple \(\reduct{L}{A}\), \(\reduct{L}{B}\), and \(\reduct{L}{C}\).
    Define \(p = f_{L, 1}^{-1} \alpha^{\mathbb{B}} f_{R,1} \cup f_{L, 2}^{-1}
    \alpha^{\mathbb{C}} f_{R,2}\). If \(L(A)\) is the domain of
    \(\reduct{L}{A}\), then the domains of the two parts in the definition of
    \(p\) intersect on \((L(A))f_{L, 1} = (L(A))f_{L, 2}\),
    \(\alpha^{\mathbb{B}}\) agrees with \(\alpha^{\mathbb{C}}\) on \(L(A)\), and
    \(f_{R, 1}\) agrees with \(f_{R, 2}\) on \( R(A) = (
    L(A))\alpha^{\mathbb{A}}\). Hence \(p \colon \mathbb{D}_L \to
    \mathbb{D}_R\) is a well-defined map.  Moreover, \( (L(A))f_{1, L} = (L(A))f_{2, L} \subseteq
    \dom(p)\), and \(f_{L, 1} p\), \(f_{L, 2} p\) are homomorphisms. Therefore
    \(p\) is a homomorphism since \(\mathbb{D}_L\) is a witness of the strong
    amalgamation property with homomorphism gluing. Let \(\mathbb{D}\) be the
    disjoint union of \(\mathbb{D}_L\) and \(\mathbb{D}_R\) with
    \(L^{\mathbb{D}}\) and \(R^{\mathbb{D}}\) being the domains of the two
    witnesses, and define \(\alpha^{\mathbb{D}}\) to be the homomorphism \(p\).
    We claim that \(\mathbb{D}\) is in \(\mathcal{K}\).  Since \(\reduct{L}{D}
    = \mathbb{D}_L\) and \(\reduct{R}{D} = \mathbb{D}_R\), it follows that
    \(\mathbb{D}\) satisfies properties (i) and (ii). Moreover,
    \(\alpha^{\mathbb{D}} \colon \reduct{L}{D} \to \reduct{R}{D}\) is a
    homomorphism as discussed above, proving the claim.  Observe that
    \(\mathcal{K}\) contains the empty \(\tau\)-structure, and so the
    amalgamation property automatically implies the joint embedding property.
    Therefore, \(\mathcal{K}\) is a Fra\"{i}ss\'{e} class.

    Let \(\mathbb{Y}\) be the Fra\"{i}ss\'{e} limit of \(\mathcal{K}\). Then
    \(\alpha^{\mathbb{Y}}\) is a homomorphism from \(\reduct{L}{Y}\) to
    \(\reduct{R}{Y}\). We will show that \(\alpha^{\mathbb{Y}}\) is
    surjective also. Suppose that \(y \in \reduct{R}{Y}\) is arbitrary, and let \(\mathbb{A}\) be the substructure of
    \(\reduct{R}{Y}\) induced by \(y\). Define \(\mathbb{B_A}\) to be a
    \(\tau\)-structure such that \(\reduct{L}{B_A}\) and \(\reduct{R}{B_A}\) are
    both isomorphic to \(\mathbb{A}\); and so that \(\alpha^{\mathbb{B_A}}\) is
    an isomorphism between the two structures. Then \(\mathbb{B_A}\) is in
    \(\mathcal{K}\), and so it embeds into \(\mathbb{Y}\). Suppose, without loss of generality, that \(\mathbb{B_A}\) is a substructure of \(\mathbb{Y}\), and let \( e \colon \reduct{R}{B_A} \to \mathbb{A}\) be an isomorphism. Then \(e\) is also a partial isomorphism between substructures of \(\mathbb{Y}\), and since \(\mathbb{Y}\)
    is homogeneous, it follows that
    \(e\) can be extended to \(h \in \Aut(\mathbb{Y})\). By the choice of \(\mathbb{B_A}\), there is \(x \in Y\) such that \( (x)\alpha^{\mathbb{Y}} = (y)e^{-1} = (y)h^{-1}\). Therefore, \(y = (x)h \alpha^{\mathbb{Y}}\) is in the image of
    \(\alpha^{\mathbb{Y}}\). 
  
  Next, we show that both \(\reduct{L}{Y}\) and \(\reduct{R}{Y}\) are
    isomorphic to \(\mathbb{X}\). 
     In order to do so, it suffices, by
    Fra\"{i}ss\'{e}'s theorem, to show that both \(\reduct{L}{Y}\) and
    \(\reduct{R}{Y}\) are homogeneous and have the same age as \(\mathbb{X}\). By definition, the age of \(\reduct{L}{Y}\) is contained in that of $\mathbb X$. 
    Using homomorphism-homogeneity, we will show that for all finite induced substructures $\mathbb A$ of \(\reduct{L}{Y}\) and all finite extensions $\mathbb B$ of $\mathbb A$ in $\Age(\mathbb X)$, there exists an embedding of $\mathbb B$ into \(\reduct{L}{Y}\) fixing $\mathbb A$ pointwise. This shows that  $\Age(\mathbb X)$ is contained in the age of \(\reduct{L}{Y}\), and that \(\reduct{L}{Y}\) is homogeneous since it immediately implies that the set of finite isomorphisms extending a given  isomorphism between finite substructures of  \(\reduct{L}{Y}\) is a back-and-forth system. The argument for \(\reduct{R}{Y}\) is the same. 
    
    So let $\mathbb A$ and $\mathbb B$ as above be given.
    Let \(\mathbb{D}\) be the substructure of
    \(\mathbb{Y}\) generated by \(A\). Since the age of \(\reduct{L}{Y}\) and 
    \(\reduct{R}{Y}\) is contained in that of \(\mathbb{X}\), we may assume that \(\reduct{L}{D}\)
    and \(\reduct{R}{D}\) are substructures of \(\mathbb{X}\). In that case,
    \(\alpha^{\mathbb{D}}\) is a homomorphism between finite 
    substructures of \(\mathbb{X}\), and so it can be extended to an endomorphism
    \(h\) of \(\mathbb{X}\), since \(\mathbb{X}\) is homomorphism-homogeneous.
    Moreover, by homogeneity, we may assume that also \(\mathbb{B}\) is a substructure of \(\mathbb{X}\). Let \(\mathbb{E}\) be the  \(\tau\)-structure obtained by taking the disjoint
    union of $\mathbb B$ and the induced substructure of \(\mathbb{X}\) on the set $(B)h$, setting \(L^{\mathbb{E}}=B\),
    \(R^{\mathbb{E}}=(B)h\), and  \(\alpha^{\mathbb{E}}\) the
    restriction of \(h\) to \(B\). Then \(\mathbb{E} \in \mathcal{K}\) and
    \(\mathbb{D}\) is an induced substructure of \(\mathbb{E}\), since \(A
    \subseteq B\) and \(\alpha^{\mathbb{E}}\) extends
    \(\alpha^{\mathbb{D}}\). Hence, by the homogeneity of $\mathbb Y$, there is an embedding  of  \(\mathbb{E}\)
    into \(\mathbb{Y}\) fixing $A$ pointwise; this map also  embeds  $\mathbb B$ into \(\reduct{L}{Y} \), proving our claim. 
    
    Let \(i_1\) and \(i_2\) denote  isomorphisms from
    \(\reduct{L}{Y}\) and \(\reduct{R}{Y}\) to \(\mathbb{X}\), respectively. The final step is to show that \(g := i_1^{-1} \alpha^{\mathbb{Y}} i_2\) is \textbf{rth}. Let
    \(U \subseteq X\) be finite, and let \(u \in X \setminus U\) and \(v \in
    X\) be such that \(g{\restriction_U}\cup \{(u, v)\}\) is a homomorphism.
    Let \(\mathbb{H}\) be a \(\tau\)-structure obtained by taking a disjoint
    union of the induced substructures of \(\mathbb{X}\) on \((U \cup \{u\})i_1^{-1}\)
    and \( ((U)g \cup \{v\})i_2^{-1}\), let these two sets be \(L^{\mathbb{H}}\) and
    \(R^{\mathbb{H}}\), and define \(\alpha^{\mathbb{H}}\) to be
    \(i_1 (g{\restriction_U}\cup \{(u, v)\}) i_2^{-1}\). Then \(\mathbb{H}\) is in
    \(\mathcal{K}\). Since \(\mathbb{Y}\) is homogeneous and since
    \(\mathbb{H}\) restricted to \((U)i^{-1}_1 \cup ((U)g \cup \{v\})i^{-1}_2\) is also in
    \(\mathcal{K}\), it follows that there is an embedding \(e \colon
    \mathbb{H} \to \mathbb{Y}\) such that \((x)e = x\) for every \((U)i^{-1}_1 \cup ((U)g \cup \{v\})i^{-1}_2\).
    If \(w := (u)i_1^{-1} e i_1\), then
    \[
      (w)g = (u)i_1^{-1} e i_1 i_1^{-1}\alpha^{\mathbb{Y}}i_2 = (u)i_1^{-1}\alpha^{\mathbb{H}} e i_2 =
      (v)i_2^{-1} e i_2 = v,
    \]
    since \(e\) is an embedding. Moreover, \(i_1^{-1} e{\restriction_{(U \cup \{u\})i_1^{-1}}} i_1\),
    that is the map  \(x \mapsto x\) if \(x \in U\) and \(x \mapsto w\) if \(x
    = u\), is a partial isomorphism of \(\mathbb{X}\). Hence there is \(\alpha
    \in \Aut(\mathbb{X})\) such that \( (x)\alpha = x\) for all \(x \in U\) and
    \( (u)\alpha g = v\).
  \end{proof}
  Hence \(\End(\mathbb{X})\) equipped with the pointwise topology has property
  \textbf{X} with respect to \(\Aut(\mathbb{X})\) by
  Theorem~\ref{prop-structures-property-x}.
\end{proof}

\section{Uniqueness and non-uniqueness}
\label{section-apex}

\subsection{Non-uniqueness of Polish topologies}
\label{subsection-self}

In \cite[Theorem 5.22]{Elliott2019aa} it was shown that the monoid $\Inj(\N)$ consisting of all injective functions from the natural numbers $\N$ to $\N$ possesses infinitely many Polish semigroup topologies containing the pointwise topology. 
We will use these distinct Polish topologies on $\Inj(\N)$ to show that certain
endomorphism monoids and monoids of self-embedding of relational structures
do not have unique Polish topologies.

\begin{prop}\thlabel{prop:non-unique-top}
  Let $S$ be a closed submonoid of $\Inj(\N)$ with the pointwise topology such that the group of units $G$ of $S$ is not closed. Then $S$ has at least two distinct Polish semigroup topologies. 
\end{prop}
\begin{proof}
  By \cite[Theorem 5.22]{Elliott2019aa},
  the pointwise topology and
  the subspace topology $\T$ induced by $\mathcal{I}_4$ (defined in \cite[Theorem  5.15]{Elliott2019aa}) are Polish semigroup topologies on $\Inj(\N)$; 
  the topology $\T$ contains the pointwise topology and these two topologies coincide on $\Sym(\N)$.
  Since $S$ is closed in the pointwise topology and $\T$, it follows that both  are Polish semigroup topologies on $S$ also. 
  The symmetric group $\Sym(\N)$ is not closed in the pointwise topology but is closed in $\mathcal{I}_4$.
  By assumption, $G$ is not closed in the pointwise topology, but
  $G = (\Sym(\N) \cap S) \cap (\Sym(\N) \cap S) ^ {-1}$ is closed in $\T$. Therefore these two topologies are distinct.
\end{proof}

If $\mathbb{X}$ is an  $\omega$-categorical structure, then $\Aut(\mathbb{X})$ is not closed in the pointwise topology on $X ^ X$; see e.g.~~\cite{Bodirsky2017aa}. Hence, Proposition~\ref{prop:non-unique-top} then implies that $\Emb(\mathbb X)$ admits at least two distinct Polish semigroup topologies. More generally, any closed submonoid of $\Inj(X)$ containing the automorphism group of  $\mathbb X$ (such as, for example, the monoid of elementary self-embeddings of $\mathbb X$, or, in the case of a \emph{model-complete core}~\cite{Cores-journal}, its endomorphism monoid) admits
at least two distinct Polish semigroup topologies.

\ignore{
  If $\mathbb{X}$ is a homogeneous relational structure, 
  then the closure $\overline{\Aut(\mathbb{X})}$ of $\Aut(\mathbb{X})$ in the pointwise topology is $\Emb(\mathbb{X})$. 
  If $\mathbb{X}$ admits an embedding which is not an automorphism, then  Proposition~\ref{prop:non-unique-top} implies that $\Emb(\mathbb{X})$ admits at least two distinct Polish semigroup topologies.
  If $\mathbb{X}$ is $\omega$-categorical, then $\Aut(\mathbb{X})$ is not closed in the pointwise topology on $X ^ X$.
  This can be demonstrated by an argument similar to, but
simpler than, the one in \thref{lem-zariski-functions1}. 
 Hence Proposition~\ref{prop:non-unique-top} implies that for any $\omega$-categorical structure $\mathbb{X}$, the self-embedding monoid $\Emb(\mathbb{X})$ admits at least two distinct Polish semigroup topologies.
 }
 \ignore{
A relational structure $\mathbb{X}$ is known as a \textit{model-complete core} if the closure of $\Aut(\mathbb{X})$ is $\End(\mathbb{X})$.  Model-complete cores play an important role in the study of computational complexity of Constraint Satisfaction Problems over $\omega$-categorical structures, see, for example,~\cite{Barto2019}. A further corollary of Proposition~\ref{prop:non-unique-top} is: if $\mathbb{X}$ is a model-complete core and $\Aut(\mathbb{X}) \neq \End(\mathbb{X})$, then $\End(\mathbb{X})$ has at least two distinct Polish semigroup topologies. }

\begin{prop}\label{prop-infinitely-many-wreath-product}
The endomorphism monoid of the graph $n \mathbb{K}_{\omega}$  without loops for every non-zero $n\in \N \cup \{\omega\}$ possesses infinitely many Polish semigroup topologies.
\end{prop}
\begin{proof}
If $N$ is any set such that $|N| = n$, then 
$\End(n\mathbb{K}_{\omega})$ is isomorphic to the wreath product $\Inj(\N) \wr N ^ N$ where $\Inj(\N)$ is the monoid of injective functions from $\N$ to $\N$, and $N ^ N$ is the monoid of all functions from $N$ to itself.
Hence, by \cite[Theorem 5.22]{Elliott2019aa}, the infinitely many Polish semigroup topologies on $\Inj(\N)$ give rise to infinitely many distinct Polish semigroup topologies on 
$\Inj(\N) \wr N ^ N$. 
\end{proof}

\subsection{Proof of  Corollary~\ref{apex}}\label{subsection-apex}

In this section we prove Corollary~\ref{apex}.
\apex*
\begin{proof}
  It follows from Theorems~\ref{theorem-the-first-theorem} and~\ref{cor-zariski-endomorphisms} and Proposition~\ref{prop-luke-9} that the pointwise topology is contained in any Polish semigroup topology on the endomorphism structures of any of the structures mentioned in the statement. On the other hand, it follows from Theorem~\ref{lem-luke-0}\eqref{lem-luke-0-iii}, Proposition~\ref{prop-xxx}, and Corollary~\ref{cor-property-x} that pointwise topology is also the maximal topology in each of the cases.
\end{proof}

\begin{theorem}\label{thm-final}
  Let \(\mathbb{X}\) be one of the following structures:
  \begin{enumerate}[\rm (i)]
      \item the random graph;
      \item the random directed graph;
      \item the graph \(\omega \mathbb{K}_n\) for $n\in \N$;
      \item the graph \(n \mathbb{K}_\omega\) with loops for $n\in \N\cup \{\omega\}$ i.e., the random equivalence relations with \(n\) countably infinite equivalence classes;
  \end{enumerate}
  as well as (i)-(iii) with all the loops included.
  Then  \(\End(\mathbb{X})\) equipped with pointwise topology has automatic 
  continuity with respect to the class of second countable topological 
  semigroups.
\end{theorem}

\begin{proof}
A Polish group $G$ has \textit{ample generics} if $G$ has a comeagre orbit when acting by conjugation on $G ^ n$ for every $n\in \N$. 
It is well known, see~\cite[Remark 5.3.8]{macpherson2011aa}, that the automorphism groups of both the random graph and the random 
  directed graph have ample generics.
It can also be shown that the automorphism group of \(\omega \mathbb{K}_n\) has ample generics, by
showing that a certain class arising from the substructures of \(\omega
\mathbb{K}_n\), for $n\in \N$, has JEP and WAP, see \cite[Theorem~6.2]{Kechris2007aa} for more details.

It was shown in  \cite[Theorem~6.24]{Kechris2007aa} that any Polish group with ample generics has automatic continuity with respect to the class of second countable topological groups. It is straightforward to use \cite[Theorem~6.24]{Kechris2007aa} to show that any such group also
has automatic continuity with respect to the class of second countable topological semigroups; see \cite[Proposition~4.1]{Elliott2019aa}. Hence
 \(\Aut(\mathbb{X})\) 
 has automatic continuity with respect to the class of second countable topological semigroups when $\mathbb{X}$ is the random graph, random directed graph, or $\omega \mathbb{K}_n$ for any $n\in \N$.
 By Proposition~\ref{prop-xxx} and Corollary~\ref{cor-property-x}(i) and (ii), the endomorphism monoids of each of the structures in parts (i) to (iii) have property \textbf{X} with respect to their automorphism groups. Thus, by \thref{lem-luke-0}(ii), these endomorphism monoids have automatic continuity with respect to the class of second countable topological semigroups. This concludes the proof of the theorem for  cases (i), (ii), and (iii) (without loops).

 Clearly, 
 the automorphism groups of the random graph, the random directed graph, and $\omega \mathbb{K}_n$ (where $n\in \N \cup \{\omega\}$) with all the loops are equal to the automorphism groups of the respective structures without loops. Both Proposition~\ref{prop-xxx} and Corollary~\ref{cor-property-x} hold for these structures with loops, and hence the proof that these structures have automatic continuity with respect to the class of second countable topological semigroups is the same as that given above. 
  
  When proving (iv), we consider the cases when $n = \omega$ and $n\neq \omega$ separately.
  It can be shown that the automorphism group of $\omega \mathbb{K}_{\omega}$ has ample generics using a similar argument as mentioned above for $\omega \mathbb{K}_n$, $n\in \N$. That $\End(\omega \mathbb{K}_{\omega})$ has automatic continuity with respect to the class of second countable topological semigroups follows by the same argument as cases (i) to (iii), via Corollary~\ref{cor-property-x}(iv).
  
  On the other hand, if $n\neq \omega$, then the automorphism group of \(n \mathbb{K}_\omega\) contains an open finite index normal subgroup and hence does not have any comeagre conjugacy classes, let alone ample generics. To show that (iv) holds, it suffices to show that for every $n\in \N\cup \{\omega\}$ the endomorphism monoid of \(n \mathbb{K}_\omega\) with loops  has property $\mathbf{X}$ with respect to a subsemigroup which has automatic continuity.
  It is shown in \cite{Elliott2019aa} that the full transformation monoid $\N ^ {\N}$ has automatic continuity with respect to the class of second countable topological semigroups. 
  
   Recall that $\End(n \mathbb{K}_{\omega})$ is topologically isomorphic to the wreath product $\N ^ \N\wr N ^ N$ where $|N|=n$, as defined before Proposition \ref{prop-infinitely-many-wreath-product}.
  If $S = \set{(f, \ldots, f, \id)}{f\in \N ^ \N}$, then clearly $S$ is isomorphic to $\N ^ \N$.
   Since $S$ is a subspace of the second countable $T_1$ space $\End(n \mathbb{K}_{\omega})$, it follows that $S$ is a second countable $T_1$ topological semigroup itself. 
 It was shown in~\cite[Theorem 5.4(v)]{Elliott2019aa} that $\N ^ \N$, and hence $S$, has a unique second countable $T_1$ semigroup topology, the pointwise topology, and hence $S$ is topologically isomorphic to $\N ^ \N$. Hence $S$ has automatic continuity with respect to the class of second countable topological semigroups. 
 
 It remains to show that $\N ^ \N\wr N ^ N$ has property \textbf{X} with respect to $S$. 
 Let $N_0, N_1, \dots$ be a partition of $\N$ into infinite sets and for every $i\in \N$ let $b_i:\N \rightarrow N_i$ be a bijection.
 For an arbitrary $s = (s_1, \ldots, s_n, t) \in \N ^ \N\wr N ^ N$ we define $f_s = (b_1, \ldots, b_n ,t) \in  \N ^ \N\wr N ^ N$ and we define $b\in \N ^ \N$ to be any function extending 
 $b_i ^ {-1}s_i : N_i \to \N$ for all $i$ such that $1\leq i \leq n$. Note that $b$ exists because the sets $N_i$ are disjoint.
If   $t_s = (b, \ldots, b, \id) \in S$, then
\[
f_s t_s = (b_1, \ldots, b_n, t)(b, \ldots, b, \id) = (b_1b, \ldots, b_nb, t) = (s_1, \ldots, s_n, t) = s.
 \]
If $B$ is any neighbourhood of $t_s$, then there exists an open neighbourhood $U = [\sigma] = \set{f\in \N ^ \N}{\sigma \subseteq f}$ of $b$ in $\N ^ \N$ such that $t_s \in U\times \cdots \times U\times \{\id\} \subseteq B$. It follows that
\[
f_s((U\times \cdots \times U \times \{\id\})\cap S) 
= \set{(b_1, \ldots, b_n, t)(u, \ldots, u, \id)}{u\in U}
= \set{(b_1u, \ldots, b_nu, t)}{u\in U}.
\]
Clearly, 
\[\set{(b_1u, \ldots, b_nu, t)}{u\in U} \subseteq \set{(k_1, \ldots, k_n, t)}{\sigma\restriction_{N_i} \subseteq b_i^{-1}k_i}\]
 and conversely, if $(k_1, \ldots, k_n, t) \in \N ^ \N \wr n ^ n$  is such that $ \sigma\restriction_{N_i}\subseteq b_i ^{-1}k_i$, then since the partial functions $b_i^{-1}k_i$ have disjoint domains, 
 there exists $u\in U$ such that $b_iu = k_i$ for every $i$. Therefore 
\[\set{(b_1u, \ldots, b_nu, t)}{u\in U} = \set{(k_1, \ldots, k_n, t)}{\sigma\restriction_{N_i} \subseteq b_i^{-1}k_i}\]
and so $f_s((U\times \cdots \times U \times \{\id\})\cap S)$ is open, as required.
\end{proof}

The countably infinite homogeneous graphs were classified, up to isomorphism, in~\cite{Lachlan1980aa} as:

    \begin{enumerate}[\rm (i)]

      \item
        the random graph;

      \item
        the random $\mathbb{K}_n$-free graph, for every $n\geq 3$;

      \item
        the graph $n \mathbb{K}_m$ where \(n, m \leq \omega\) and at least one of \(n\) or \(m\) is equal to \(\omega\).
    \end{enumerate}
    and the duals of these graphs. At this point, we have almost achieved a classification of those homogeneous graphs whose endomorphism monoid has a unique Polish semigroup topology. More precisely, the monoids of endomorphisms of each of the following homogeneous graphs has a unique Polish semigroup topology by Corollary~\ref{apex}: the random graph;
        the graph $n \mathbb{K}_m$ where \(n, m \leq \omega\) and at least one of \(n\) or \(m\) is equal to \(\omega\).
On the other hand, the endomorphism monoids of the random  $\mathbb{K}_n$-free  graph, for every $n\geq 3$ have at least two distinct Polish semigroup topologies. This follows from Proposition~\ref{prop:non-unique-top}, since every endomorphism of the \(\mathbb{K}_n\)-free universal graph is an embedding. We show this is the case when \(n = 3\). Suppose that \(x, y\) are non-adjacent vertices in the \(\mathbb{K}_3\)-free graph which get mapped to an edge by some endomorphism \(f\). Then there is a vertex \(z\) which is adjacent to both \(x\) and \(y\), and so the triple \( (x)f\), \((y)f\), \((z)f\) forms a triangle, which is impossible. A similar argument shows that \(f\) has to be injective, and thus \(f\) is an embedding.
        
We do not know whether the endomorphism monoids of the duals of the $\mathbb{K}_n$-free universal graphs or the duals of the graphs $n \mathbb{K}_m$ where \(n, m \leq \omega\) and at least one of \(n\) or \(m\) is equal to \(\omega\) have a unique Polish semigroup topology or not. 

\begin{question}
  The endomorphism monoids of which homogeneous graphs have a unique Polish topology?
\end{question}

\ignore{
\begin{question}
  The endomorphism monoid of which  random $\mathbb{K}_n$-free graphs have property \textbf{X} with respect to their automorphism groups and/or a unique Polish topology?
\end{question}
}


\section{Another wee foray into the land of clones}\label{section-clones}

In this section we extend some of the results from earlier in the paper to clones. We
begin by giving the relevant definitions. 
If $C$ is any set, $g: C ^ m \to C$ for some $m\in \N\setminus \{0\}$ and $f_1,\ldots, f_m: C ^ n \to C$ for some $n\in \N\setminus\{0\}$, then we define $(f_1,\ldots, f_m) \circ_{m, n} g : C ^ n \to C$ by
\[
  (x_1,\ldots, x_n) \mapsto \left( (x_1,\ldots, x_n)f_1, \ldots, (x_1, \ldots, x_n)f_m \right)g.
\]
For $n\in \N\setminus\{0\}$ and  $i \in \{1, \ldots, n\}$, we denote by $\pi ^ n_i : C ^ n \to C$ the \textit{$i$-th projection of $C ^ n$}  defined by
\[
  (x_1, \ldots, x_n) \mapsto x_i.
\]

A \emph{(function) clone $\mathscr{C}$} with domain $C$ is a set of functions of finite
arity from $C$ to $C$ which is closed under composition  and also contains all
projections. 
 More precisely, the following hold:
 \begin{enumerate}[(i)]
   \item if $g \in \mathscr{C}$ is $m$-ary and $f_1,\ldots, f_m \in
           \mathscr{C}$ are $n$-ary, then
         $(f_1,\ldots, f_m) \circ_{m,n} g\in \mathscr{C}$;
 
   \item $\pi ^ n_i\in \mathscr{C}$ for every $n\in \N\setminus\{0\}$ and every $i \in \{1, \ldots, n\}$.
 \end{enumerate}
If $\mathscr{C}$ is a function clone  and $\T$ is a topology on $\mathscr C$, 
then we will say that
$\T$ is \textit{topological} for $\mathscr{C}$, or $\T$ is a \textit{clone topology} on $\mathscr{C}$, if the composition of functions $\circ_{m, n}$ are continuous for every $m, n\in\N\setminus\{0\}$.
The set of all finite arity functions from a set $X$ to itself is called the \emph{the full function clone on  $X$}; denoted by $\mathscr{O}_X$. We do not permit nullary functions in $\mathscr{O}_X$ following \cite{Bodirsky2017aa,kerkhoff2014short}.
We define the \emph{topology of pointwise convergence} on
$\mathscr{O}_X$ in a similar way to the full transformation
monoid $X ^ X$ where the subbasic open sets are of the form
\[
  U_{(a_1, \dots, a_n),b}=\{ f \in \mathscr{O}_X \colon (a_1, \ldots, a_n)f = b\}
\]
for every $a_1, \ldots, a_n, b \in X$. Similarly to the monoid case, if $X$ is countably infinite, this topology is Polish and topological for the full function clone $\mathscr{O}_X$.
A \textit{polymorphism} of a structure $\mathbb{X}$ is a homomorphism from a finite positive power of $\mathbb{X}$ into $\mathbb{X}$. The set of all polymorphisms of $\mathbb{X}$ forms a function clone on the set $X$ which is closed in $\mathscr{O}_X$. This clone is called the \textit{polymorphism clone of $\mathbb{X}$} and denoted by $\Poly(\mathbb{X})$.
A \textit{clone homomorphism} is a map between function clones which preserves arities, maps projections to corresponding projections, and preserves the composition maps \(\circ_{m, n}\) for every \(n\) and \(m\).
A topological clone $\mathscr{C}$ is said to have \emph{automatic continuity} with
respect to a class of topological clones, if every homomorphism from $\mathscr{C}$ to a
member of that class is continuous. 

We extend the results of the preceding sections to clones by associating a semigroup to each clone. 
If $\mathscr{C}$ is a function clone which is topological with respect to
the topology $\mathcal{T}$, then we define the associated semigroup $S_{\mathscr{C}}$ by:
the elements of $S_{\mathscr{C}}$ are the functions in
$\mathscr{C}$ with multiplication  given by
\[
  (x_1,\ldots, x_n) f * g = ((x_1, \ldots, x_n)f, \ldots,
  (x_1, \ldots, x_n)f)g
\]
for all $x_1, \ldots, x_n \in C$, whenever $f$ is an $n$-ary function in
$\mathscr{C}$. 
If both $f$ and $g$ are unary functions, then the operation $*$ is
the usual composition of functions. Since the composition of functions in $\mathscr{C}$ is continuous with respect to the topology  $\mathcal{T}$, it follows that $\mathcal{T}$ is a semigroup topology on  $S_{\mathscr{C}}$.
 We remark that  every clone homomorphism from one function clone to another function clone is also a homomorphism between the corresponding semigroups, but the converse does not hold; in other words, the algebraic structure of $S_{\mathscr{C}}$ is weaker than that of the clone $\mathscr C$. However, as it will turn out, this  weaker semigroup structure is sufficient to determine the
topology in the cases we consider here.

\begin{lem}[cf. Lemma~7.1 in \cite{Elliott2019aa}]\label{lem:lift-ac}
  Let \(\mathscr{C}\) be a topological clone, and suppose that \(S_\mathscr{C}\) has automatic continuity with respect to the topology on \(\mathscr{C}\) and the class of second countable topological semigroups. Then \(\mathscr{C}\) has automatic continuity with respect to the class of second countable topological clones.
\end{lem}

The first of the main results in this section is the following analogue of Theorem~\ref{theomodeltheorypropertyx}.

\begin{theorem}\label{theomodeltheorypropertyx-clone}
  Let \(\mathbb{X}\) be a relational structure which is homogeneous and
  homomorphism-homogeneous such that the age of \(\mathbb{X}\) is closed under finite non-empty direct products and has the strong
  amalgamation property with homomorphism gluing. Then \(S_{\Poly(\mathbb{X})}\)
  equipped with the pointwise topology has property \textbf{X} with respect to
  \(\Aut(\mathbb{X})\).
\end{theorem}
\begin{proof}
The proof of this theorem is similar to the proof of Theorem~\ref{theomodeltheorypropertyx}, and for the sake of brevity is omitted. However, we will outline the key differences between the two proofs and what changes need to be made to the definitions.

\begin{itemize}
\item[$\triangleright$] Extend Definitions \ref{prop-structures-iii} and \ref{prop-structures-ii} of \textbf{fo} and \textbf{back} to elements $f\in \Poly(\mathbb{X})$.

\item[$\triangleright$] Change the statements of Theorem~\ref{prop-structures-property-x} and Claim~\ref{claim-inside-proof} as follows.  Instead of a single \textbf{back} and \textbf{fo} endomorphism $f$, demand a \textbf{back} and \textbf{fo} polymorphism $f_n$ with domain $\mathbb{X}^n$ for every $n$.

\item[$\triangleright$] In the proof of Claim~\ref{claim-inside-proof}, define $\mathbb{Y}_0$ to be $\mathbb{X}^n$ instead of $\mathbb{X}$. Note that, in the same proof, Age($\mathbb{Y})=\text{Age}(\mathbb{X})$ because Age($\mathbb{X}$) is closed under finite direct products and everything in the age of $\mathbb{X}$ is built in the construction of $\mathbb{Y}$ by extending the empty substructure of $\mathbb{X}^n$.
\qedhere
\end{itemize}
\end{proof}
\begin{lem}\label{min-end-to-clone}
Suppose that \(\mathbb{X}\) is a   structure, and for all \(x_1, x_2, \ldots, x_n\in X\) there exists \(a \in X\) and \(g_1, g_2, \ldots, g_n\in \End(\mathbb{X})\) with \((a)g_i= x_i\) for each \(i\in\{1,\ldots,n\}\). If the minimum Hausdorff semigroup topology on \(\End(\mathbb{X})\) is the pointwise topology, then the minimum Hausdorff clone topology on \(\Poly(\mathbb{X})\) is the pointwise topology.
\end{lem}
\begin{proof}
Suppose that \(\Poly(\mathbb{X})\) is equipped with some Hausdorff clone topology. By assumption, the set
\[U_{x, y}:= \set{f\in \End(\mathbb{X})}{(x)f= y}\]
is open for all \(x, y\in X\). Let \(x_1, x_2, \ldots, x_n, y\in X\) and write \(\mathbf{x}= (x_1, x_2, \ldots, x_n)\). It suffices to show that the set
\[U_{\mathbf{x}, y}:= \set{f\in \Pol(\mathbb{X})}{f \text{ has arity } n,\ (\mathbf{x})f= y}\]
is open. By assumption we can choose \(a\in X\) and \(g_1, g_2, \ldots, g_n\) such that \((a)g_i= x_i\) for each \(i\in\{1,\ldots,n\}\). That \(U_{\mathbf{x}, y}\) is open follows from 
\[f\in U_{\mathbf{x}, y}\iff (g_1, g_2, \ldots, g_n) \circ_{n, 1} f\in U_{a, y}.\qedhere\]
\end{proof}
\begin{theorem}\label{thm-unique-clones}
The pointwise topology is the unique Polish clone topology on the polymorphism clones of the following structures:
    \begin{enumerate}[\rm (i)]
      \item the random graph;
      \item the random directed graph;
      \item the random strict partial order;
      \item the random equivalence relation, i.e., the graph \(\omega \mathbb{K}_\omega\) with loops;
  \end{enumerate}
  as well as (i)-(iii) with all the loops included.
\end{theorem}

\begin{proof}
  Let $\mathbb{X}$ be any of the structures in the statement of the theorem. As discussed in the proof of  Corollary~\ref{cor-property-x},  $\mathbb{X}$ is homogeneous, homomorphism-homogeneous, and has the strong amalgamation property with homomorphism gluing. Moreover, in each case, it is easy to show that Age($\mathbb{X}$) is closed under finite direct products. Hence \(S_{\Pol(\mathbb{X})}\) has property \textbf{X} with respect to $\Aut(\mathbb{X})$ by Theorem~\ref{theomodeltheorypropertyx-clone}.
  It follows from Theorem~\ref{lem-luke-0}(i) that the pointwise topology is maximal among Polish clone topologies on \(S_{\Pol(\mathbb{X})}\), and thus the Polish clone topologies on \(\Pol(\mathbb{X})\).
  
  On the other hand, it follows from Theorems~\ref{theorem-the-first-theorem} and~\ref{cor-zariski-endomorphisms} and Lemma~\ref{min-end-to-clone} that the pointwise topology is the minimal Polish clone topology on \(\Pol(\mathbb{X})\). Hence the pointwise topology is the unique Polish clone topology on \(\Pol(\mathbb{X})\).
\end{proof}

\begin{theorem}\label{thm-ac-clones}
  Let \(\mathbb{X}\) be one of the following structures:
  \begin{enumerate}[\rm (i)]
      \item the random graph;
      \item the random directed graph;
 \item the random equivalence relation, i.e., the graph \(\omega \mathbb{K}_\omega\) with loops;
  \end{enumerate}
  as well as (i)-(ii) with all the loops included.
  Then  \(\Pol(\mathbb{X})\) equipped with pointwise clone topology has automatic 
  continuity with respect to the class of second countable topological 
  clones.
\end{theorem}
\begin{proof}
  The proof of this theorem is analogous to that of Theorem~\ref{thm-final}.
  As discussed in the proof of Theorem~\ref{thm-unique-clones}, if $\mathbb{X}$ is any of the structures in the statement, then \(S_{\Pol(\mathbb{X})}\) has property \textbf{X} with respect to $\Aut(\mathbb{X})$.
  As mentioned in the proof of Theorem~\ref{thm-final}, $\Aut(\mathbb{X})$ has automatic continuity with respect to the class of second countable topological semigroups. Theorem~\ref{lem-luke-0}(ii) implies that $\Pol(\mathbb{X})$ has automatic continuity for the class of second countable topological clones, as required.
\end{proof}
\bibliography{sim}{}

\begin{thebibliography}{10}

\bibitem{Barbina2007}
Silvia Barbina and Dugald Macpherson.
\newblock Reconstruction of homogeneous relational structures.
\newblock {\em Journal of Symbolic Logic}, 72(3):792--802, 2007.

\bibitem{Behrisch2017aa}
Mike Behrisch, John~K. Truss, and Edith Vargas-Garc\'{\i}a.
\newblock Reconstructing the topology on monoids and polymorphism clones of the
  rationals.
\newblock {\em Studia Logica}, 105(1):65--91, 2017.

\bibitem{Cores-journal}
Manuel Bodirsky.
\newblock Cores of countably categorical structures.
\newblock {\em Logical Methods in Computer Science ({LMCS})}, 3(1):1--16, 2007.

\bibitem{Bodirsky2018aa}
Manuel Bodirsky, David Evans, Michael Kompatscher, and Michael Pinsker.
\newblock A counterexample to the reconstruction of {$\omega$}-categorical
  structures from their endomorphism monoid.
\newblock {\em Israel Journal of Mathematics}, 224(1):57--82, 2018.

\bibitem{Bodirsky2014ab}
Manuel Bodirsky and Michael Pinsker.
\newblock Topological {B}irkhoff.
\newblock {\em Transactions of the American Mathematical Society},
  367(4):2527--2549, 2014.

\bibitem{Bodirsky2017aa}
Manuel Bodirsky, Michael Pinsker, and Andr\'{a}s Pongr\'{a}cz.
\newblock Reconstructing the topology of clones.
\newblock {\em Transactions of the American Mathematical Society},
  369(5):3707--3740, 2017.

\bibitem{Bodirsky2014}
Manuel Bodirsky, Michael Pinsker, and András Pongrácz.
\newblock Projective clone homomorphisms.
\newblock {\em Journal of Symbolic Logic}, 86(1):148--161, 2021.

\bibitem{Cameron1990}
Peter~J. Cameron.
\newblock {\em Oligomorphic Permutation Groups}.
\newblock Cambridge University Press, June 1990.

\bibitem{Cameron1999aa}
Peter~J. Cameron.
\newblock {\em Permutation groups}, volume~45 of {\em London Mathematical
  Society Student Texts}.
\newblock Cambridge University Press, Cambridge, 1999.

\bibitem{Cameron2006aa}
Peter~J. Cameron and Jaroslav Ne{\v{s}}et{\v{r}}il.
\newblock Homomorphism-homogeneous relational structures.
\newblock {\em Combin. Probab. Comput.}, 15(1-2):91--103, 2006.

\bibitem{Chang2017aa}
Xiao Chang and Paul Gartside.
\newblock Minimum topological group topologies.
\newblock {\em J. Pure Appl. Algebra}, 221(8):2010--2024, 2017.

\bibitem{Cohen2016aa}
Michael~P. Cohen and Robert~R. Kallman.
\newblock {${\rm PL}_+(I)$} is not a {P}olish group.
\newblock {\em Ergodic Theory Dynam. Systems}, 36(7):2121--2137, 2016.

\bibitem{Coleman:2017aa}
Thomas Coleman.
\newblock {\em Automorphisms and endomorphisms of first-order structures}.
\newblock PhD thesis, University of East Anglia, 2017.

\bibitem{Elliott2019aa}
Luke Elliott, Julius Jonušas, Zak Mesyan, James~D. Mitchell, Micha\l{}
  Morayne, and Yann~H. P\'eresse.
\newblock Automatic continuity, unique {P}olish topologies, and {Z}ariski
  topologies for monoids and clones, 2019, arXiv:1912.07029.

\bibitem{Evans1990}
David~M. Evans and Paul~R. Hewitt.
\newblock Counterexamples to a conjecture on relative categoricity.
\newblock {\em Annals of Pure and Applied Logic}, 46(2):201--209, February
  1990.

\bibitem{Fraisse2000aa}
Rolland {Fra\"{\i}ss\'e}.
\newblock {\em {Theory of relations. Transl. from the French by P. Clote. With
  an appendix by Norbert Sauer. Revised ed.}}, volume 145.
\newblock Amsterdam: North-Holland, revised ed. edition, 2000.

\bibitem{Gartside2008ab}
Paul Gartside and Bojana Peji{\'c}.
\newblock Uniqueness of {P}olish group topology.
\newblock {\em Topology Appl.}, 155(9):992--999, 2008.

\bibitem{Gaughan1967aa}
Edward~D. Gaughan.
\newblock Topological group structures of infinite symmetric groups.
\newblock {\em Proc. Nat. Acad. Sci. U.S.A.}, 58:907--910, 1967.

\bibitem{givant2008introduction}
Steven Givant and Paul Halmos.
\newblock {\em Introduction to Boolean algebras}.
\newblock Springer Science \& Business Media, 2008.

\bibitem{Herwig1998}
Bernhard Herwig.
\newblock Extending partial isomorphisms for the small index property of many
  $\omega$-categorical structures.
\newblock {\em Israel Journal of Mathematics}, 107(1):93--123, December 1998.

\bibitem{Hodges1997aa}
Wilfrid Hodges.
\newblock {\em A shorter model theory}.
\newblock Cambridge University Press, Cambridge, 1997.

\bibitem{Hodges1993ab}
Wilfrid Hodges, Ian Hodkinson, Daniel Lascar, and Saharon Shelah.
\newblock The small index property for $\omega$-stable ($\omega$-categorical
  structures and for the random graph.
\newblock {\em Journal of the London Mathematical Society}, s2-48(2):204--218,
  October 1993.

\bibitem{Hrushovski1992aa}
Ehud Hrushovski.
\newblock Extending partial isomorphisms of graphs.
\newblock {\em Combinatorica}, 12(4):411--416, 1992.

\bibitem{Kallman1976aa}
Robert~R. Kallman.
\newblock A uniqueness result for topological groups.
\newblock {\em Proc. Amer. Math. Soc.}, 54:439--440, 1976.

\bibitem{Kallman1979aa}
Robert~R. Kallman.
\newblock A uniqueness result for the infinite symmetric group.
\newblock In {\em Studies in analysis}, volume~4 of {\em Adv. in Math. Suppl.
  Stud.}, pages 321--322. Academic Press, New York-London, 1979.

\bibitem{Kallman1984aa}
Robert~R. Kallman.
\newblock A uniqueness result for a class of compact connected groups.
\newblock In {\em Conference in modern analysis and probability ({N}ew {H}aven,
  {C}onn., 1982)}, volume~26 of {\em Contemp. Math.}, pages 207--212. Amer.
  Math. Soc., Providence, RI, 1984.

\bibitem{Kallman1984ab}
Robert~R. Kallman.
\newblock Uniqueness results for the {$ax+b$} group and related algebraic
  objects.
\newblock {\em Fund. Math.}, 124(3):255--262, 1984.

\bibitem{Kallman1986aa}
Robert~R. Kallman.
\newblock Uniqueness results for homeomorphism groups.
\newblock {\em Trans. Amer. Math. Soc.}, 295(1):389--396, 1986.

\bibitem{Kallman2010aa}
Robert~R. Kallman and Alexander~P. McLinden.
\newblock The {P}oincar\'e and related groups are algebraically determined
  {P}olish groups.
\newblock {\em Collect. Math.}, 61(3):337--352, 2010.

\bibitem{Kechris1995aa}
Alexander~S. Kechris.
\newblock {\em Classical descriptive set theory}, volume 156 of {\em Graduate
  Texts in Mathematics}.
\newblock Springer-Verlag, New York, 1995.

\bibitem{Kechris2007aa}
Alexander~S. Kechris and Christian Rosendal.
\newblock Turbulence, amalgamation, and generic automorphisms of homogeneous
  structures.
\newblock {\em Proc. Lond. Math. Soc. (3)}, 94(2):302--350, 2007.

\bibitem{kerkhoff2014short}
Sebastian Kerkhoff, Reinhard P{\"o}schel, and Friedrich~Martin Schneider.
\newblock A short introduction to clones.
\newblock {\em Electr. Notes Theor. Comput. Sci.}, 303:107--120, 2014.

\bibitem{Lachlan1980aa}
Alistair~H. Lachlan and Robert~E. Woodrow.
\newblock Countable ultrahomogeneous undirected graphs.
\newblock {\em Trans. Amer. Math. Soc.}, 262(1):51--94, 1980.

\bibitem{Lascar1991}
Daniel Lascar.
\newblock Autour de la propri{\'{e}}t{\'{e}} du petit indice.
\newblock {\em Proceedings of the London Mathematical Society},
  s3-62(1):25--53, January 1991.

\bibitem{macpherson2011aa}
Dugald Macpherson.
\newblock A survey of homogeneous structures.
\newblock {\em Discrete Mathematics}, 311(15):1599 -- 1634, 2011.
\newblock Infinite Graphs: Introductions, Connections, Surveys.

\bibitem{melleray2008some}
Julien Melleray.
\newblock Some geometric and dynamical properties of the {U}rysohn space.
\newblock {\em Topology and its Applications}, 155(14):1531--1560, 2008.

\bibitem{paolini2018strong}
Gianluca Paolini and Saharon Shelah.
\newblock The strong small index property for free homogeneous structures,
  2018, 1703.10517.

\bibitem{Paolini2019}
Gianluca Paolini and Saharon Shelah.
\newblock Reconstructing structures with the strong small index property up to
  bi-definability.
\newblock {\em Fundamenta Mathematicae}, 247(1):25--35, 2019.

\bibitem{Paolini2020}
Gianluca Paolini and Saharon Shelah.
\newblock Automorphism groups of countable stable structures.
\newblock {\em Fundamenta Mathematicae}, 248(3):301--307, 2020.

\bibitem{Pech2016aa}
Christian Pech and Maja Pech.
\newblock On automatic homeomorphicity for transformation monoids.
\newblock {\em Monatsh. Math.}, 179(1):129--148, 2016.

\bibitem{Pech2017aa}
Christian Pech and Maja Pech.
\newblock Reconstructing the topology of the elementary self-embedding monoids
  of countable saturated structures.
\newblock {\em Studia Logica}, 106(3):595--613, September 2017.

\bibitem{Pech:2018aa}
Christian Pech and Maja Pech.
\newblock Polymorphism clones of homogeneous structures: gate coverings and
  automatic homeomorphicity.
\newblock {\em Algebra universalis}, 79(2):35, 2018.

\bibitem{Perez2020aa}
Jos\'e Perez and Carlos Uzcategui.
\newblock Topologies on the symmetric inverse semigroup, 2020,
  arXiv:2012.03041.

\bibitem{Rado1964aa}
Richard Rado.
\newblock Universal graphs and universal functions.
\newblock {\em Acta Arith.}, 9:331--340, 1964.

\bibitem{Rosendal2007ac}
Christian Rosendal and S\l{}awomir Solecki.
\newblock Automatic continuity of homomorphisms and fixed points on metric
  compacta.
\newblock {\em Israel J. Math.}, 162:349--371, 2007.

\bibitem{Rubin1994}
Matatyahu Rubin.
\newblock On the reconstruction of $\aleph_0$-categorical structures from their
  automorphism groups.
\newblock {\em Proceedings of the London Mathematical Society},
  s3-69(2):225--249, September 1994.

\bibitem{Sabok2019aa}
Marcin Sabok.
\newblock {A}utomatic continuity for isometry groups.
\newblock {\em J. Inst. Math. Jussieu}, 18(3):561--590, 2019.

\bibitem{Shelah1984aa}
Saharon Shelah.
\newblock Can you take {S}olovay's inaccessible away?
\newblock {\em Israel Journal of Mathematics}, 48(1):1--47, 1984.

\bibitem{Solovay1970aa}
Robert~M. Solovay.
\newblock A model of set-theory in which every set of reals is lebesgue
  measurable.
\newblock {\em Annals of Mathematics}, 92(1):1--56, 1970.

\bibitem{Truss1989}
John~K. Truss.
\newblock Infinite permutation groups {II}. subgroups of small index.
\newblock {\em Journal of Algebra}, 120(2):494--515, February 1989.

\end{thebibliography}
\bibliographystyle{hplain}

\end{document}